\documentclass[11 pt]{article}
\usepackage{graphicx}
\usepackage{amsmath,latexsym,amssymb,amsthm,enumerate,amscd}

\theoremstyle{plain}
\newtheorem{theorem}{Theorem}[section]
\newtheorem{proposition}[theorem]{Proposition}
\newtheorem{corollary}[theorem]{Corollary}
\newtheorem{lemma}[theorem]{Lemma}
\theoremstyle{definition}
\newtheorem{definition}[theorem]{Definition}
\newtheorem{algorithm}[theorem]{Algorithm}

\newtheorem{remark}[theorem]{Remark}
\newtheorem{example}[theorem]{Example}

\theoremstyle{plain}

\newtheorem{notation}[theorem]{Notation}

\numberwithin{equation}{section}
\numberwithin{table}{section} 
\def\<{\left<}
\def\>{\right>}
\def\k{\mathsf{k}}

\renewcommand{\thefootnote}{\arabic{footnote}}

\title{The smallest part of the generic partition of the nilpotent commutator of a nilpotent matrix}
\author{Leila Khatami} 

\begin{document}
\date{}
\maketitle
{\let\thefootnote\relax}
{\let\thefootnote\relax}

\begin{abstract}
Let $\k$ be an infinite field. Fix a Jordan nilpotent $n \times n$ matrix $B=J_P$ with entries in $\k$ and associated Jordan type $P$. Let $Q(P)$ be the Jordan type of a generic nilpotent matrix commuting with $B$. In this paper, we use the combinatorics of a poset associated to the partition $P$, to give an explicit formula for the smallest part of $Q(P)$, which is independent of the characteristic of $\k$. This, in particular, leads to a complete description of $Q(P)$ when it has at most three parts.
\end{abstract}

\section*{Introduction}
Let $\k$ be an infinite field and $\mathcal{M}at_n(\k)$ be the set of all $n \times n$ matrices with entries in $\k$. Suppose that $B=J_P \in\{ \mathcal{M}at_n(\k)\}$ is a Jordan matrix with associated Jordan type -- Jordan block partition-- $P\vdash n$. Recall that the centralizer and the nilpotent centralizer of $B$ are defined as follows.
$$\begin{array}{ll}
\mathcal{C}_B=&\{A \in \mathcal{M}at_n(\k) \, | \, AB=BA \}, \\
\mathcal{N}_B=&\{A \in \mathcal{C}_B \, | \, A^n=0 \}.
\end{array}$$

It is well known that $\mathcal{N}_B$ is an irreducible algebraic variety (see \cite[Lemma 1.5]{BI}). Therefore, there is a unique partition of $n$ corresponding to the Jordan type of a generic element of $\mathcal{N}_B$. We denote this unique partition by $Q(P)$. The map $P \to Q(P)$ has been studied by different authors (see \cite{BI}, \cite{BIK}, \cite{IK}, \cite{uniqueness}, \cite{KO}, \cite{Oblak}, \cite{Pa}). 

The number of parts of the partition $Q(P)$ was completely determined by R. Basili (\cite[Proposition 2.4]{Basili03} and \cite[Theorem 2.17]{BIK}). It is known that if char$\k=0$ (see \cite{KO}) or char $\k>n$ (see \cite{BI}), then the map $P \to Q(P)$ is idempotent: $Q(Q(P))=Q(P)$.  In \cite{Oblak}, P. Oblak gives a formula for the index-- largest part -- of the partition $Q(P)$ over the field of real numbers. Her result was extended to any infinite field, by  A. Iarrobino and the author in  \cite{IK}. This, in particular, gives rise to an explicit formula for the parts of $Q(P)$ when it has one or two parts. 

In this paper, we work with a poset $\mathcal{D}_P$ determined by a given partition $P$. The poset was used in \cite{KO} and \cite{Oblak}, and was defined in \cite{BIK} in close connection with $\mathcal{U}_B$, a maximal nilpotent subalgebra of the centralizer $\mathcal{C}_B$. In Section~\ref{Dsec}, we review the preliminaries. We first define the poset $\mathcal{D}_P$ and then recall the classic partition invariant $\lambda(P)=\lambda(\mathcal{D}_P)$ of the poset $\mathcal{D}_P$, defined in terms of the lengths of unions of chains in $\mathcal{D}_P$, and also the partition $\lambda_U(P)$ associated to the poset $\mathcal{D}_P$, defined and studied in \cite{IK} and \cite{uniqueness}.

In Section~\ref{min} (Theorem~\ref{Oblak min}) we define an invariant $\mu(P)$ of the partition $P$ and show that it is equal to  the smallest part of the partition $\lambda_U(P)$. In Section~\ref{min of lambda}, enumerating disjoint maximum antichains in $\mathcal{D}_P$, we prove that the smallest part of $\lambda(P)$ is also equal to $\mu(P)$ (Theorem~\ref{spread min part}). 

In section~\ref{Qmin}, using the main result of \cite{IK}, we prove that the smallest part of $Q(P)$ is the same as the smallest part of $\lambda(P)$ (Theorem~\ref{Q min part}). As a consequence we obtain an explicit formula independent of characteristic of $\k$ for $Q(P)$ when it has at most three parts (Corollary~\ref{3 parts}).
\\ 
\noindent{\bf Acknowledgement.} The author is grateful to A. Iarrobino for invaluable discussions on the topic, as well as for his comments and suggestions on the paper. The author is also thankful to Bart Van Steirteghem and Toma\v{z} Ko\v{s}ir for their thorough comments on an earlier version of this paper\footnote{This paper is a much revised version of results and proofs outlined in the last two sections of a preprint circulated early in 2012. (See arXiv:1202.6089v1.)}.

\section{Poset $\mathcal{D}_P$}\label{Dsec}

Let $n$ be a positive integer and  $P$ a partition of $n$. For a positive integer $p$, we denote by $n_p\geq 0$ the number of parts of size $p$ in $P$.

Assume that $B=J_P$, is the $n \times n$ Jordan matrix with Jordan type $P$, $\mathcal{C}_B$ its  centralizer and $\mathcal{J}$, the Jacobson radical of $\mathcal{C}_B$. Write $P=(p_s^{n_s}, \cdots, p_1^{n_1})$ with $p_s>\cdots>p_1$ and $n_i>0$, for all $i$. Then $\mathcal{C}_B/\mathcal{J}$ is semisimple, and is here isomorphic to the matrix algebra $\mathcal{M}_{n_s}(\k)\times \cdots \times \mathcal{M}_{n_1}(\k)$: there is a projection $\pi:\mathcal{C}_B\to \mathcal{M}_{n_s}(\k)\times \cdots \times \mathcal{M}_{n_1}(\k).$ (See \cite[Lemma 2.3]{Basili03}, \cite[Theorem 2.3]{BIK}, \cite[Theorem 6]{HW}.)

Note that $\mathcal{N}_B=\pi^{-1}(\mathcal{N}_s \times \cdots \times \mathcal{N}_1)$, where $\mathcal{N}_i$ is the subvariety of nilpotent matrices in $\mathcal{M}_{n_i}(\k)$. Let $\mathcal{U}_B=\pi^{-1}(\mathcal{U}_s \times \cdots \times \mathcal{U}_1)$, where $U_i$ is the subalgebra of $ \mathcal{M}_{n_i}(\k)$ consisting of strictly upper triangular matrices. It is easy to see that $\mathcal{U}_B$ is a maximal nilpotent subalgebra of $\mathcal{C}_B$ and that for any element $N \in \mathcal{N}_B$, there exists a unit $C \in \mathcal{C}_B$ such that $CNC^{-1} \in \mathcal{U}_B$ (see \cite[Lemma 2.2]{BIK}). Thus to find the Jordan partition of a generic element of $\mathcal{N}_B$ it is enough to study a generic element of $\mathcal{U}_B$.

Recall that if $v$ and $v'$ are elements of a poset, we say $v'$ {\it covers} $v$ if $ v<v'$ and there is no element $v''$ in the poset such that $v<v''< v'$. For a partition $P$ of $n$, we define a poset $\mathcal{D}_P$ through its covering edge diagram; a digraph with vertices corresponding to the elements of $\mathcal{D}_P$ and an edge from $v$ to $v'$ in if and only if $v'$ covers $v$ in $\mathcal{D}_P$. This poset, originally defined in \cite{BIK}, is  closely related to the algebra $\mathcal{U}_B$. In fact, the paths in the digraph of $\mathcal{D}_P$ correspond to the non-zero entries of a generic element of  $\mathcal{U}_B$. The study of  combinatorial properties of this poset has led to results on the partition $Q(P)$ (see \cite{Oblak}\footnote{Although the poset $\mathcal{D}_P$ was formally introduced later in \cite{BIK}, P. Oblak implicitly worked with it using a slightly different setting in \cite{Oblak}.}, \cite{BIK}, \cite{IK} and \cite{uniqueness}).

\begin{definition}\label{poset}

Let $P$ be a partition of $n$ and write $P=(p_s^{n_{p_s}}, \cdots, p_1^{n_{p_1}})$ such that $p_s>\cdots>p_1$  and $n_{p_i} > 0$ for $1 \leq i \leq s.$ The associated poset $\mathcal{D}_P$ is defined as follows. 

As a set, the poset $D_P$ consists of $n$ elements, each labeled by a triple $(u,p,k)\in \mathbb{N}^3$ such that $p\in\{p_1, \cdots, p_s\}$, $1\leq u \leq p$ and $1\leq k\leq n_p$. The covering relations in $\mathcal{D}_P$ are as follows:  $(v,p_j,\ell) \mbox{ covers } (u,p_i,k) \mbox{ if and only if one of the following holds.}$
\begin{itemize}
\item[($\alpha$)] If $j=i+1$, $v=u+p_{i+1}- p_{i}$, $\ell=1$ and $k=n_{p_i}$, 
\item[($\beta$)] If $j=i-1$, $v=u$, $\ell=1$ and $k=n_{p_i}$, 
\item[($e$)] If $j=i$, $v=u$ and $\ell=k+1$, or
\item[(${\omega}$)] If $j=i$,  $p_{i+1}-p_i>1$, $p_i-p_{i-1}>1$, $v=u+1$, $\ell=1$ and $k=n_{p_i}$.
\end{itemize}

We associate to this poset its covering edge digraph, which we also denote by $\mathcal{D}_P$ and we visualize it as follows.

$\bullet$ {\bf Vertices of }$\mathcal{D}_P$\\
For each $1 \leq i \leq t$, there are $n_{p_i}$ rows,  each consisting of $p_i$ vertices labeled by triples $(u, p_i, k)$ such that the first and last components of the triple are increasing when we go from left to right and from bottom to top, respectively. For each $1\leq i \leq s$, we refer to the set of all vertices $(u, p_i, k)$ as {\it level} $p_i$ (or the $i$-th level from bottom) of $\mathcal{D}_P$. (See Figure 1.)

$\bullet$ {\bf Covering edges of }$\mathcal{D}_P$\\
Each edge of the digraph, corresponding to a covering relation in the poset,  is one of the following.
\begin{itemize}
\item[$(\alpha)$] For $1\leq i< s$, and each $1\leq u \leq p_i$, the directed edge $\alpha_{i, i+1}(u)$  from the top vertex $(u,p_i,n_{p_i})$ of the $u$-th column of level $p_i$ to the bottom vertex $(u+p_{i+1}- p_{i}, p_{i+1}, 1)$ of the $(u+p_{i+1}- p_{i})$-th column in level $p_{i+1}$.

\item[$(\beta)$] For $1<i \leq s$, the directed edge $\beta_{p_{i}, p_{i-1}}(u)$ from the top vertex $(u,p_i,n_{p_i})$ of the $u$-th column of level $p_i$ to the bottom vertex $(u, p_{i-1}, 1)$ in the $u$-th column of level $p_{i-1}$. 

\item[$(e)$] For $1 \leq i \leq s$, $1 \leq u \leq p_i$ and $1 \leq k < n_{p_i}$, the upward directed edge $e_{{p_i,k}(u)}$ from $(u,p_i, k)$ to $(u,p_i, k+1)$ .

\item[$(\omega)$] For any isolated $p_i$ ({\it i.e.} $p_{i+1}-p_i>1$ and $p_i-p_{i-1}>1$) and any $1 \leq u <p_i$, the directed edge $\omega_{p_i}(u)$ from $(u, p_i, n_{p_i})$ to $(u+1, p_i, 1)$. (See Figure 2.)
\end{itemize}

\end{definition}
\begin{figure}[htb]
\begin{center}\label{ver}
\includegraphics[scale=.65]{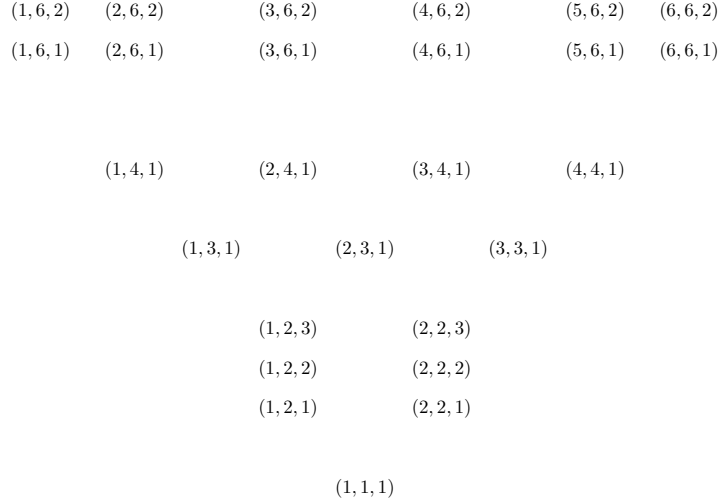}
\end{center}
\vspace{-3.5in}
\caption{Labeled vertices of the poset $\mathcal{D}_P$ for $ P=(6^2,4,3,2^3,1)$}
\end{figure}

\begin{figure}[htb]
\vspace{-.3in}
\begin{center}\label{edge}
\includegraphics[scale=.5]{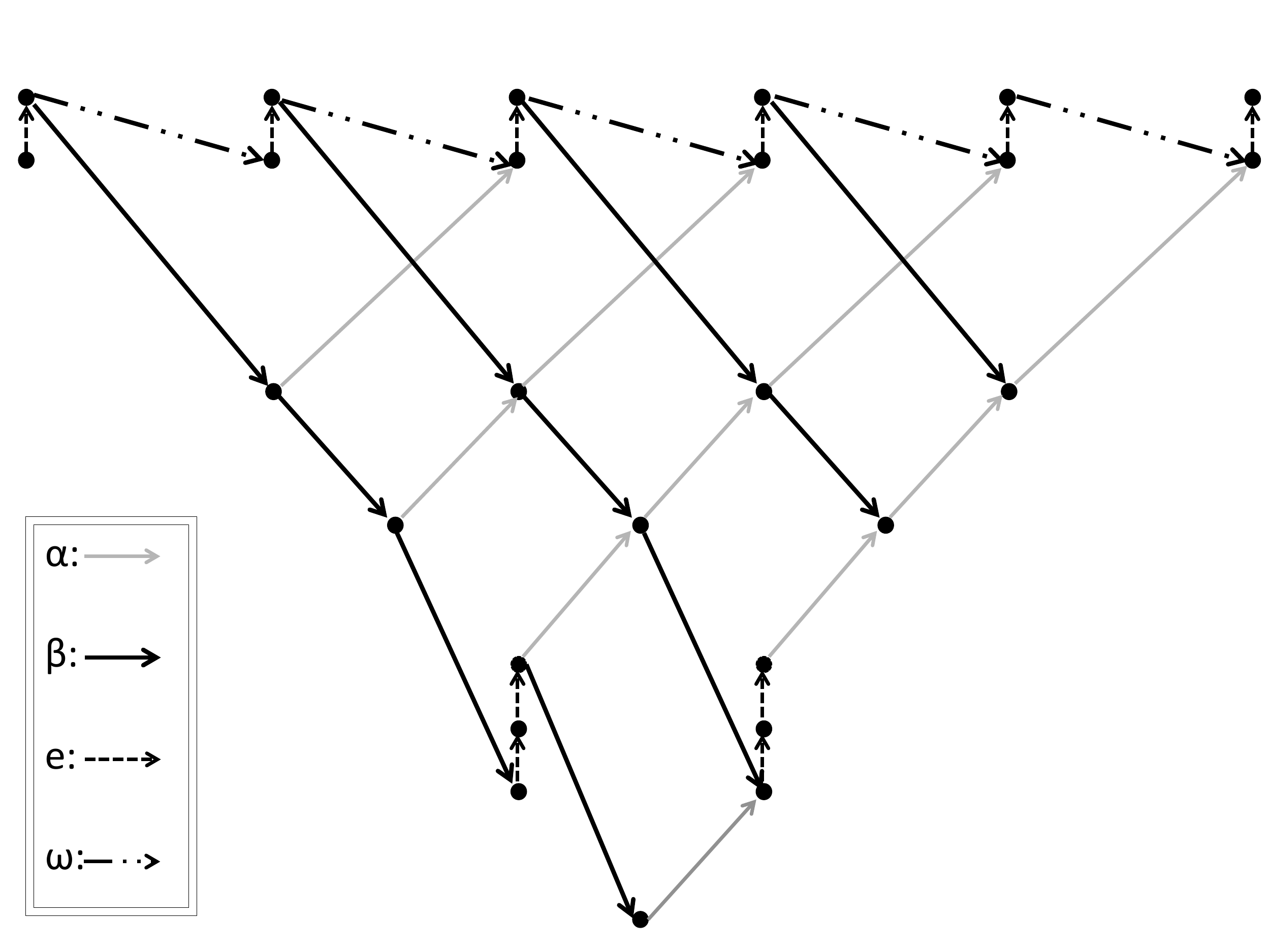}
\end{center}
\vspace{-0.5 in}
\caption{Covering edges in the poset $\mathcal{D}_P$ for $ P=(6^2,4,3,2^3,1)$}
\end{figure}

\begin{remark}\label{compare}
Let $(x,p,k)$ and $(y,q,\ell)$ be two elements of the poset $\mathcal{D}_P$. Then $(x,p,k) \leq (y,q,\ell)$ if and only if there exists a directed path in the diagram of $\mathcal{D}_P$ that goes from vertex $(x,p,k)$ to vertex $(y,q,\ell)$. Thus, by Definition~\ref{poset} we have  
$$\begin{array}{lll}
(x,p,k)\leq (y,q, \ell) \Leftrightarrow
&p<q \mbox{ and } q-p+x \leq y,\mbox{ or}\\
&p\geq q \mbox{ and } x\leq y,\mbox{ or}\\
&p=q, \, \, x=y \mbox{ and } k\leq \ell.
\end{array}$$
\end{remark}
Recall that for a partition $P$, the partition $Q(P)$ is defined as the Jordan type of a generic nilpotent matrix commuting with $B=J_P$. We associate two other partitions to $P$ in terms of its poset $\mathcal{D}_P$. 

First, the partition $\lambda(P)$, the classical partition associated to $\mathcal{D}_P$. Recall that a {\it chain} in a poset is defined as a totally ordered subset, whose {\it length} is its cardinality. In a finite poset, a chain is called {\it maximum} if there is no chain in the poset with a strictly larger cardinality.

\begin{definition}\label{lambda of P}
Let $P$ be a partition and $\mathcal{D}_P$ the associated poset. The partition $\lambda(P)=(\lambda_1, \lambda_2 \cdots )$ is defined such that $\lambda_k=c_k-c_{k-1}$, where $c_k$ is the maximum cardinality of a union of $k$ chains in $\mathcal{D}_P,$ for $k=0,1\cdots.$ 
\end{definition}

Next, we recall the definition and properties of the partition $\mathcal{\lambda}_U(P)$, associated to $\mathcal{D}_P$ with a similar process as in the definition of $\lambda(P)$, but only considering certain types of chains closely related to almost rectangular subpartitions of $P$ (see \cite{uniqueness}). Recall that an {\it almost rectangular} partition is a partition whose largest and smallest parts differ by at most 1.

\begin{definition}\label{general U chain}
Let $P$ be a partition of $n$. For a positive integer $r$ and a set $\frak{A}=\{a_1,a_1+1, \cdots , a_r,a_r+1\}\subset\mathbb{N}$ such that $a_1<a_1+1 <\cdots< a_r< a_{r}+1$, we define the {\it $r$-$U$-chain} $U_{\frak{A}}$ as follows: 
$$\begin{array}{c}
U_{\frak{A}}=\cup_{i=1}^{r}S_{{\frak{A}};i},\mbox{ where}\\
S_{{\frak{A}};i}=\{(u,p,k)\in \mathcal{D}_P\,\mid \, p\in\{a_i, a_i+1\} \mbox{ and } i\leq u \leq p-i+1\}\\
\cup \{(u,p,k)\in \mathcal{D}_P\,\mid \, p>a_i+1 \mbox{ and } u\in\{i,p-i+1\} \}.
\end{array}$$
Note that each $S_{{\frak{A}},i}$ is a chain in $\mathcal{D}_P$ and that $S_{{\frak{A}},i}\cap S_{{\frak{A}},j}=\emptyset$ if $i\neq j$. A $1$-$U$-chain is called a {\it simple $U$-chain}. A {\it maximum} simple $U$-chain is a simple $U$-chain with the maximum cardinality among all simple $U$-chains in $\mathcal{D}_P.$
\end{definition}

\begin{figure}[htb]\label{U_chain}
\begin{tabular}{cc}
\includegraphics[scale=.29]{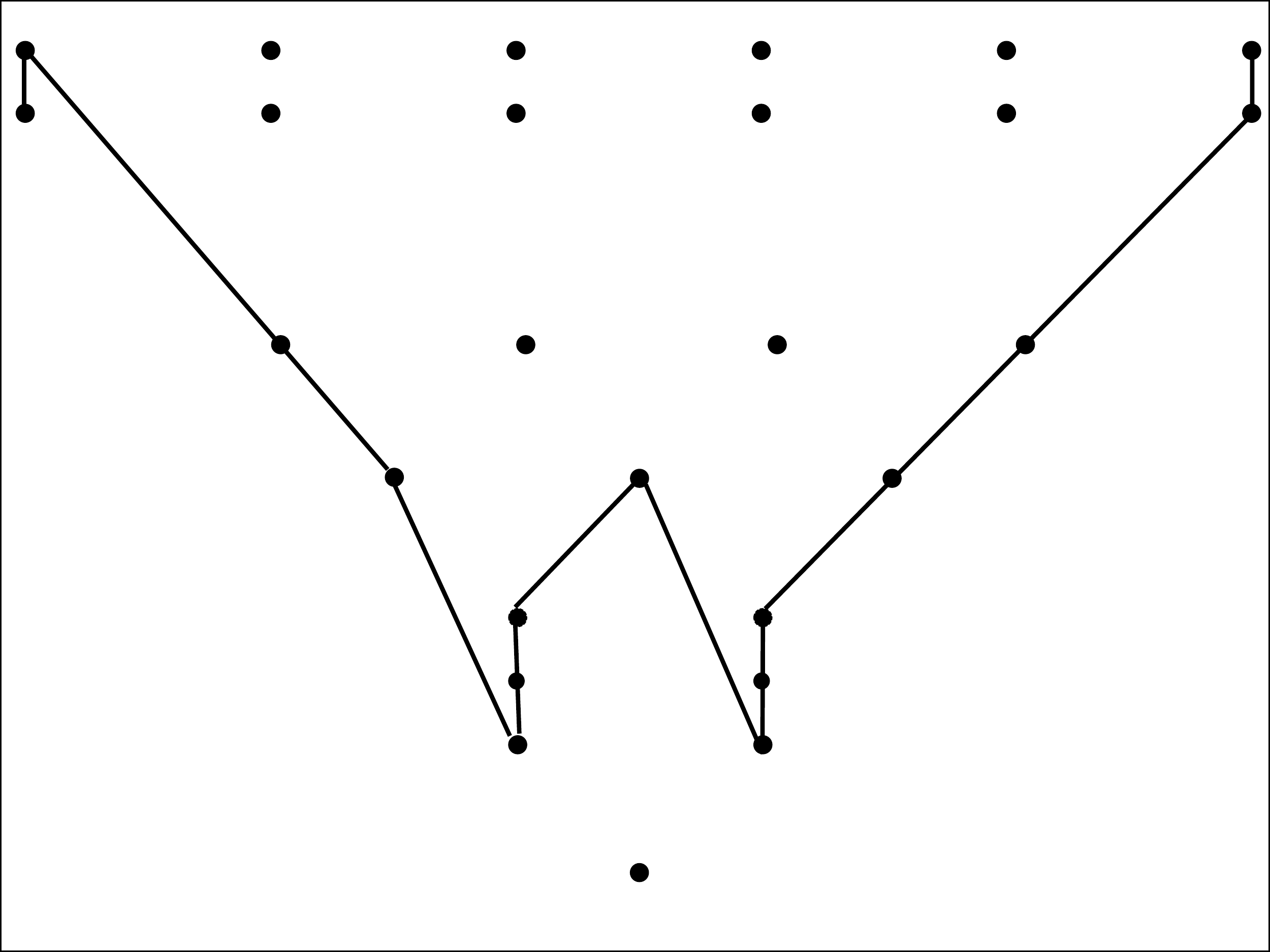}&\includegraphics[scale=.29]{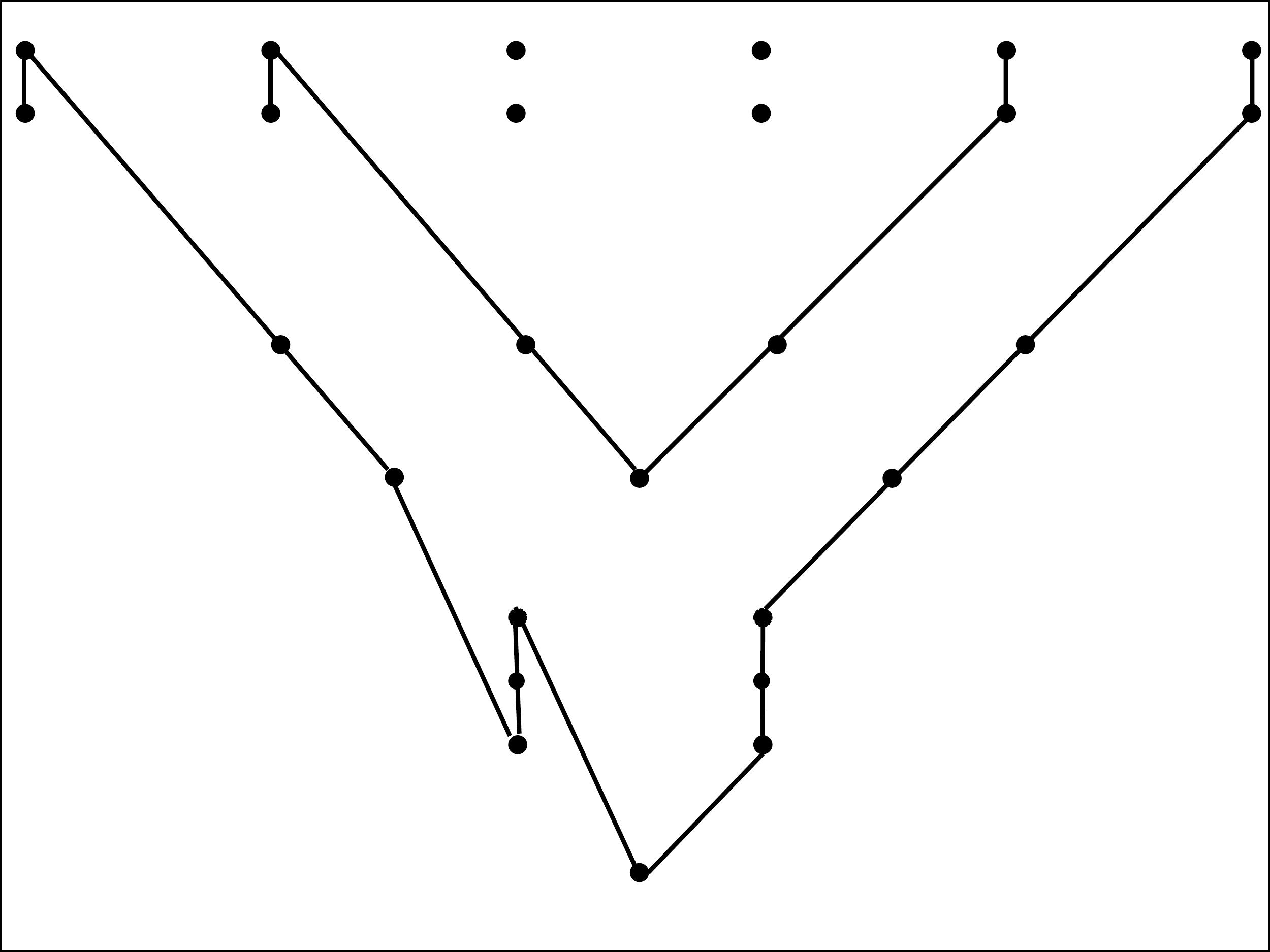}
\end{tabular}
\caption{$U_{\{2,3\}}$ and $U_{\{1,2,3,4\}}$  in the poset $\mathcal{D}_P$ for $ P=(6^2,4,3,2^3,1)$}
\end{figure}

By the definition above, if $\frak{A}=\{a,a+1\}$, then the length of the corresponding simple $U$-chain $U_{\frak{A}}$ is 
\begin{equation}\label{simple}
\mid U_\frak{A}\mid=a\, n_{a}+(a+1)\,n_{a+1}+2\, \displaystyle{\sum_{p>a+1}n_p}.
\end{equation}

Now we are ready to define $\lambda_U(P)$, the third partition associated to $P$.
\begin{definition}\label{lambda U}
Let $P$ be a partition and $\mathcal{D}_P$ the associated poset. The partition $\lambda_U(P)=(\lambda_{U,1}, \lambda_{U,2}, \cdots)$ is defined such that $\lambda_{U,k}=u_{k}-u_{k-1},$ where $u_k$ is the maximum cardinality of a $k$-$U$-chain in $\mathcal{D}_P$, for $k=0, 1, \cdots$.

\end{definition}

\begin{example} Let $P=(6^2,4,3,2^3,1)$. Then 
$$\begin{array}{l}
u_1=15=|U_{\{1,2\}}|=|U_{\{2,3\}}|,\\
u_2=23=|U_{\{1,2,5,6\}}|=|U_{\{2,3,5,6\}}|, \mbox{ and }\\
u_3=26=|U_{\{1,2,3,4,5,6\}}|=|\mathcal{D}_P|.
\end{array}$$

So we have $\lambda_U(P)=(15,8,3).$
\end{example}

\begin{definition}\label{dominance order}{\bf(Dominance partial order)}
Let $P$ and $Q$ be partitions of $n$. Here we write $P=(p_1, p_2, \cdots)$ and  $Q=(q_1, q_2, \cdots)$ with $p_1\geq p_2\geq \cdots$ and $q_1\geq q_2\geq \cdots$. Then $P \leq Q$ if and only if for any $k\in \mathbb{N}$, we have $\displaystyle{\sum_{i=1}^{k}p_i \leq \sum_{i=1}^{k}q_i}.$
\end{definition}
\bigskip

It is obvious from the definitions that for any partition $P$, $\lambda_U(P) \leq \lambda(P).$ In \cite[Theorem 3.9]{IK}, we prove that the inequality $\lambda_U(P) \leq Q(P)$ also holds.

\section{Smallest part of $\lambda_U(P)$}\label{min}

In this section we define a numerical invariant, $\mu(P)$, associated to a partition $P$, and show that it is the smallest part of the partition $\lambda_U(P)$. In the next two sections we prove that the smallest part of the partition $\lambda(P)$ and also the smallest part of $Q(P)$ are equal to $\mu(P)$ as well.

\begin{definition}
Let $s$ be a positive integer. A partition $P=(p_s^{n_s}, \cdots, p_1^{n_1})$ of $n$ such that $p_1>0$, $p_{i+1}=p_{i}+1$, for $1\leq i < s$ and $n_i>0$ for $1\leq i \leq s$, is called an $s$-spread of $n$, or simply a {\it spread}. If $s\leq 2$ then $P$ is called an {\it almost rectangular} partition. In other words, $P$ is almost rectangular if its biggest and smallest part differ by at most 1.

\end{definition}
\begin{notation}
For a rational number $a\in \mathbb{Q}$, $\lceil a \rceil$ denotes the smallest integer greater than or equal to $a$.

\end{notation}
\begin{definition}\label{rP}
For an arbitrary partition $P$, let $r(P)$ denote the minimum number of almost rectangular subpartitions of $P$ appearing in a decomposition of $P$ as a union of almost rectangular subpartition. Note that if $P$ is an $s$-spread then $r(P)=\lceil\frac{s}{2}\rceil$.
\end{definition}

By \cite[Proposition 2.4]{Basili03} and \cite[Theorem 2.17]{BIK}, for any partition $P$, the partition $Q(P)$ has exactly $r(P)$ parts. By Definitions~\ref{general U chain} and~\ref{lambda U}, it is also clear that $\lambda_U(P)$ has $r(P)$ parts. In Section~\ref{min of lambda}, we will show that $\lambda(P)$ has $r(P)$ parts as well. 

For a partition $P$, we write $\lambda_U(P)=(\lambda_{U,1}(P), \cdots, \lambda_{U,r(P)}(P))$ as in Definition~\ref{lambda U}. In particular, $\min(\lambda_U(P))=\lambda_{U,r(P)}(P).$

The following proposition gives a relation between the smallest part of a partition and the smallest parts of its subpartitions. We will use this in Definition~\ref{mu defn} and in the proof of Theorem~\ref{Oblak min}. 

\begin{notation}
Let $P$ be a partition and $\ell$ a positive integer smaller than the smallest part of $P$. We write $P-\ell$ to denote the partition obtained from $P$ by subtracting $\ell$ from each part of it. 
\end{notation}
\begin{proposition}\label{U min combine}
Let $P$ and $Q$ be two partitions. Let $P=(p_s^{n_s}, \cdots, p_1^{n_1})$ such that $p_s>\cdots >p_1$ and $n_i>0$ for all $1\leq i \leq s$, and $Q=(q_t^{m_t}, \cdots, q_1^{m_1})$ such that $q_t> \cdots >q_1$ and $m_i>0$ for all $1\leq i \leq t$. If $q_1 \geq p_s+2$ and $R=(Q,P)$ then $$\min(\lambda_U(R))=\min\{\min(\lambda_U(P)),\min(\lambda_U(Q-2r(P)))\}.$$ 
\end{proposition}

\begin{proof}
Let $\bar{Q}=Q-2r(P)$. Throughout the proof, we identify elements of $\mathcal{D}_{\bar{Q}}$ with elements of $\mathcal{D}_R$ via the relabeling map $\iota:\mathcal{D}_{\bar{Q}}\to \mathcal{D}_R$ defined by $\iota((u,p,k))=(u+r(P),p+2r(P),k).$

First note that since $q_1\geq p_s+2$, we have $r(R)=r(P)+r(Q)$ and $$q_1-2r(P)\geq p_s+2-2r(P)\geq p_1+s+1-2r(P)\geq p_1 >0. $$ Thus $r(Q)=r(\bar{Q})$ and therefore $r(R)=r(P)+r(\bar{Q}).$

Assume that $1\leq k \leq r(Q)$ and let $U_{\bar{\frak{A}}}$ be a $k$-$U$-chain in $\mathcal{D}_{\bar{Q}}$. Let $\frak{B}=\{a+2r(P)\,|\, a\in \bar{\frak{A}}\}\cup \{p_1, \cdots, p_{2r(P)}\}$. Then $U_{\frak{B}}$ is an $(r(P)+k)$-$U$-chain in $\mathcal{D}_R$ that contains $\mathcal{D}_P$. We also have $\mathcal{D}_R\setminus U_\frak{B}=\mathcal{D}_{\bar{Q}}\setminus U_{\bar{\frak{A}}}.$ So if $U_{\bar{\frak{A}}}$ has maximum cardinality among all $k$-$U$-chains in $\mathcal{D}_{\bar{Q}}$, then by Definition~\ref{lambda U} we get

\begin{equation}
\displaystyle{\sum_{i=r(P)+k+1}^{r(P)+r(Q)}\lambda_{U,i}(R)}\, \leq \, \mid \mathcal{D}_R\setminus U_\frak{B}\mid = \mid \mathcal{D}_{\bar{Q}}\setminus U_{\bar{\frak{A}}}\mid\,  = \displaystyle{\sum_{i=k+1}^{r(Q)}\lambda_{U,i}(\bar{Q})}.
\end{equation}

In particular, setting $k=r(Q)-1$, we get $\min(\lambda_U(R))\leq \min(\lambda_U(Q)).$

On the other hand, any $\ell$-$U$-chain $U$ in $\mathcal{D}_P$ can be extended to an $\ell$-$U$-chain $\bar{U}$ in $\mathcal{D}_R$ by adding the first and the last $r(P)$ vertices of each row of $\mathcal{D}_{Q}.$ So if $\frak{Q}=\{q_1, \cdots, q_{2r(Q)}\}$ then $\bar{U}\cup U_\frak{Q}$ is an $(\ell+r(Q))$-$U$-chain in $\mathcal{D}_R$ that contains $\mathcal{D}_{Q}$, and we also have $$\mathcal{D}_R\setminus(\bar{U}\cup U_\frak{Q})=\mathcal{D}_P\setminus U.$$ Thus, if $U$ has maximum cardinality among all $\ell$-$U$-chains in $\mathcal{D}_P$ then we get
\begin{equation}\sum_{i=\ell+r(Q)+1}^{r(P)+r(Q)}\lambda_{U,i}(R) \leq \mid \mathcal{D}_R\setminus(\bar{U}\cup U_\frak{Q})\mid = \mid \mathcal{D}_P\setminus U \mid = \displaystyle{\sum_{i=\ell+1}^{r(P)}\lambda_{U,i}(P)}.\end{equation} In particular, setting $\ell=r(P)-1$, we get $\min(\lambda_U(R))\leq \min(\lambda_U(P)).$

Thus \begin{equation} \min(\lambda_U(R)) \leq \min\{\min(\lambda_U(P)),\min(\lambda_U(\bar{Q}))\}.\end{equation}

Conversely, suppose that $U_\frak{A}$ is an $(r(R)-1)$-$U$-chain in $\mathcal{D}_R$ with maximum cardinality. Since $r(R)=r(P)+r(Q)$, either $\mathcal{D}_P\subseteq U_\frak{A}$ or $\mathcal{D}_Q\subseteq U_\frak{A}$.

If $\mathcal{D}_P\subseteq U_\frak{A}$, then $(\mathcal{D}_{\bar{Q}} \cap U_\frak{A})$ is an $(r(Q)-1)$-$U$-chain in $\mathcal{D}_{\bar{Q}}$. So, by Definition~\ref{lambda U} and the choice of $U_\frak{A}$, we have $\min(\lambda_U(\bar{Q})) \leq |\mathcal{D}_{\bar{Q}}\setminus U_\frak{A}|=|\mathcal{D}_R\setminus U_\frak{A}|=\min(\lambda_U(R)).$

If $\mathcal{D}_Q\subseteq U_\frak{A}$, then $(\mathcal{D}_P \cap U_\frak{A})$ is an $(r(P)-1)$-$U$-chain in $\mathcal{D}_P$. So $$\min(\lambda_U(P)) \leq |\mathcal{D}_P\setminus U_\frak{A}|=|\mathcal{D}_R\setminus U_\frak{A}|=\min(\lambda_U(R)).$$

This completes the proof of the proposition.
\end{proof}

We now define our key invariant $\mu(P)$ that we will show is equal to the smallest part of all three partitions $\lambda_U(P)$, $\lambda(P)$, and $Q(P)$.  

\begin{definition}\label{mu defn}
Let $P$ be an $s$-spread with the smallest part $p$, so we can write $P=\left((p+s-1)^{n_s}, \cdots, p^{n_1}\right)$, such that $n_i>0$ for $1\leq i\leq s$. Setting $n_{s+1}=0$, we define \begin{equation}\label{mu  spread equi}\mu(P)=\min \{p \, n_{2i-1}+(p+1) n_{2j} \, | \, 1 \leq i \leq j \leq r(P) \}.\end{equation}

Note that if $s$ is an odd number, then $r(P)=\frac{s+1}{2}$ and therefore $n_{2r(P)}=0$. Thus 
\begin{equation}\label{mu od}
\mu(P)=p\cdot \min \{ \, n_{2i-1} \, | \, 1 \leq i \leq r(P) \}.
\end{equation}

Now let $P$ be an arbitrary partition. We can write $P=P_\ell \, \cup \cdots \cup \, P_1$ such that each $P_k$ is a spread and the biggest part of $P_k$ and is less than or equal to the smallest part of $P_{k+1}$ minus two. 
For $k=1, \cdots, \ell$, let $\bar{r}_k=\displaystyle{\sum_{i=1}^{k-1}r(P_i)}$. We define \begin{equation}\mu(P)=\min\{\mu(P_k-2\bar{r}_k)\,|\, 1\leq k \leq \ell\}.\end{equation}

\end{definition}

The following lemma is an immediate consequence of Definition~\ref{mu defn}.

\begin{lemma}\label{mu union} 
Let $P$ and $P'$ be two partitions such that the largest part of $P$ is smaller than the smallest part of $P'$ minus 1. Then $\mu(P'\, \cup P)=\min\{\mu(P), \mu(P'-2r(P))\}$.
\end{lemma}

\begin{remark}
Note that in general $\mu(P)-r\neq \mu(P-r).$ For example, let $P=(4,3^2)$ and $r=2.$ Then $\mu(P)=10$ and $\mu(P-r)=\mu((2,1^2))=4$. 
\end{remark}

\begin{example}
Let $P=(11,10,9^2,8,6,5,3,2,1^2)$. Then $P=(P_3,P_2,P_1)$, where $$P_1=(3,2,1^2), 
P_2=(6,5), \mbox{ and }
P_3=(11,10,9^2,8).$$
We have $r(P_1)=2$,  $r(P_2)=1$ and $r(P_3)=2$ and therefore $\bar{r}_1=0$, $\bar{r}_2=2$ and $\bar{r}_3=3$. Thus  $\mu(P)=\min\{\mu(P_1), \mu(P_2-4), \mu(P_3-6)\}$. 

By definition
$$\begin{array}{lll}
\mu(P_1)&=\mu((3,2,1^2))&=(1) \cdot \min\{2,1\}=1, \\
\mu(P_2-4)&=\mu((2,1))&=\min\{(1)(1)+(2)(1)\}=3, \mbox{ and} \\
\mu(P_3-6)&=\mu((5,4,3^2,2))&=\min\{(2)(1)+(3)(2), (2)(1)+(3)(1), (2)(1)+(3)(1)\}= 5.
\end{array}$$
Therefore $\mu(P)=1$.

\end{example}

\begin{theorem}\label{Oblak min}
For any partition $P$ of $n$,  $\min(\lambda_U(P))=\mu(P).$
\end{theorem}

\begin{proof}

Write $P=P_\ell \, \cup \cdots \cup \, P_1$ such that for each $k$, $1\leq k \leq \ell$, $P_k$ is a spread and the biggest part of $P_k$ is less than or equal to the smallest part of $P_{k+1}$ minus two for $k=1, \cdots, \ell$. By Definition~\ref{mu defn} of the invariant $\mu$ and Proposition~\ref{U min combine}, it is enough to prove the equality for a spread.

For the rest of the proof we assume that $P=(p_s^{n_s}, \cdots, p_1^{n_1})$ with $n_i>0$ for $i=1, \cdots, s$ is an $s$-spread. So we have $p_i=p_1+i-1$ and $r=r(P)=\lceil s/2 \rceil$. 

Let $U_\frak{A}$ be an $(r-1)$-$U$-chain in $\mathcal{D}_P$ with maximum cardinality. By definition of $\lambda_U(P)$, $$\min(\lambda_U(P)) =n-|U_\frak{A}|=|\mathcal{D}_P\setminus \frak{A}|.$$ Assume that $\frak{A}=\{a_1,a_1+1, \cdots, a_{r-1}, a_{r-1}+1\}$ such that $p_1\leq a_1<a_1+1<\cdots<a_{r-1}\leq p_{s}$. Let $\frak{B}=\{p_1,\cdots,p_s\}\setminus\frak{A}.$

\noindent {\bf Case 1.} $s$ is odd. 

In this case $r=\frac{s+1}{2}$ and $|\frak{A}|=2(r-1)=s-1$. Thus $|\frak{B}|=1.$ If $\frak{B}=\{p_\ell\}$ then, by definition of $U_\frak{A}$ (see Definition~\ref{general U chain}), $\ell=2i-1$, for some $1\leq i \leq r$. Thus $$\mathcal{D}_P \setminus U_\frak{A} =\{(u,p_{2i-1},k)\,|\, i\leq u\leq p_{2i-1}-i+1 \mbox{ and } 1\leq k \leq n_{2i-1}\}.$$ 

Since $P$ is a spread, $p_{2i-1}=p_1+2i-1$, and therefore by maximality of the size of $U_{\frak{A}}$, we get $$\begin{array}{ll}
\min(\lambda_U(P))&= |\mathcal{D}_P \setminus U_\frak{A}| \\&= p_1 \, \min \{n_{2i-1} \, | \, 1 \leq i \leq m\}\\&=\mu(P).\end{array}$$So the proof is complete in this case.

\noindent {\bf Case 2.} $s$ is even.

In this case $|\frak{B}|=2$. Suppose that $\frak{B}=\{p_\ell, p_{\ell'}\}$ such that $\ell < \ell'$. By definition of $U_\frak{A}$, $\ell$ must be odd and $\ell'$ must be even.  If $\ell=2i-1$ and $\ell'=2j$ for $i,j$ such that $1\leq i\leq j \leq r$, then
$$\begin{array}{ll}
\mathcal{D}_P \setminus U_\frak{A} =&\{(u,p_{2i-1},k)\,|\, i\leq u\leq p_{2i-1}-i+1 \mbox{ and } 1\leq k \leq n_{2i-1}\}\\
&\cup\{(v,p_{2j},k')\,|\, j\leq u\leq p_{2j}-j+1 \mbox{ and } 1\leq k' \leq n_{2j}\}.
\end{array}$$

Once again, we use the fact that $P$ is a spread and that $U_\frak{A}$ is an $(r-1)$-$U$-chain with maximum cardinality to get $$\begin{array}{ll}
\min(\lambda_U(P))&=|\mathcal{D}_P \setminus U_\frak{A}|\\&=\min\{p_1n_{2i-1}+(p_1+1)n_{2j}\,|\, 1\leq i\leq j\leq r\}\\&=\mu(P).\end{array}$$

This completes the proof.
\end{proof}
\section{Smallest part of $\lambda(P)$}\label{min of lambda}
In this section we study the partition $\lambda(P)$ (see Definition~\ref{lambda of P}) and determine its number of parts as well as its smallest part.

Let $\mathcal{D}$ be a poset of cardinality $n$. Recall that an {\it antichain} in $\mathcal{D}$ is a subset in which no two elements are comparable. For $k \in \{0, 1, \cdots, n\}$, let $a_k$ denote the maximum cardinality of a union of $k$ antichains in $D$. Let $\tilde{\lambda}_k=a_k-a_{k-1}$ for all $k \geq 1$ and define $\tilde{\lambda}(\mathcal{D})=(\tilde{\lambda}_1, \tilde{\lambda_2}, \cdots)$. In \cite{Greene}, C. Greene proves the following theorem.

\begin{theorem}\label{Greene}(\cite{Greene}, see\cite[Theorem 2.1]{Britz-Fomin})\label{conjugate}
Let $\mathcal{D}$ be a finite poset of cardinality $n$. Then $\lambda(\mathcal{D})$ and $\tilde{\lambda}(\mathcal{D})$ are conjugate partitions of $n$.
\end{theorem}

This implies that $\tilde{\lambda}_1$ is the number of parts of $\lambda(\mathcal{D})$. It also shows that the smallest part of $\lambda(\mathcal{D})$ is equal to the multiplicity of part $\tilde{\lambda_1}$ in $\tilde{\lambda}$. In other words, $\min(\lambda(D))$ is equal to the number of disjoint maximum antichains in $\mathcal{D}.$ It is this number that we will determine for the poset $\mathcal{D}_P$ (Theorem~\ref{spread min part}).

\begin{lemma}\label{antichain condition}
Let $P=(p_s^{n_s}, \cdots, p_1^{n_1})$ be a partition of $n$, such that $p_s>\cdots>p_1$ and $n_i>0$ for $1 \leq i \leq s$. 

\begin{itemize} 

\item[(a)] Let  $p_1 \leq p<q\leq p_s$ and assume that $(x,p,k)$ and $(y,q,\ell)$ are two elements in $\mathcal{D}_P$. Then  $\{\, (x,p,k), (y,q,\ell)\, \}$ is an antichain if and only if $x<y<q-p+x$.

\item[(b)] Assume that $A$ and $B$ are antichains in $\mathcal{D}_P$. Let $p=\max\{p_i\,|\, (x_i,p_i,k_i) \in A\}, \mbox{ and }
q=\min\{q_i\,|\, (y_i,q_i,\ell_i) \in B\}$.

If $p<q$, $(x,p,k)\in A$ and $(y,q,\ell) \in B$, then $A\, \cup \,B$ is an antichain if and only if $\{(x,p,k), (y,q,\ell)\}$ is an antichain.
\end{itemize}
\end{lemma}

\begin{proof}

{\bf Part (a)} This is a consequence of Remark~\ref{compare}.
\bigskip

{\bf Part (b)} If $A \, \cup \, B$ is an antichain then $\{(x,p,k), (y,q,\ell)\}\subseteq A\, \cup \,B$ is obviously an antichain. 

Now assume that $\{(x,p,k), (y,q,\ell)\}$ is an antichain. By Part (a), we have \begin{equation}x<y<q-p+x.\end{equation}

If $(x', p', k')\in A$ is such that $(x', p', k')\neq (x, p, k)$, then by choice of $p$, $p'<p$ and, by Part (a), we must have \begin{equation}x'<x<p-p'+x'.\end{equation} Similarly, if $(y', q', \ell')\in B$ is such that $(y', q', \ell')\neq (y, q, \ell)$, then $q'>q$ and \begin{equation}y<y'<q'-q+y.\end{equation} 

Using Part (a) once again, we can conclude that each element of $A$ is incomparable to each element of $B$, as desired.

\end{proof}

\begin{definition}\label{O and E}

Suppose that $P=(p_s^{n_s}, \cdots, p_1^{n_1})$ is an $s$-spread. Let $r=r(P)$ and $0 \leq t \leq r$. 

For $1\leq x \leq p_1$ and ${\bf k}=(k_1, \cdots, k_t)$ such that $1 \leq k_i \leq n_{2i-1}$, we define  $$O^t(x, {\bf k}) = \, \{(x+i-1,p_{2i-1},k_i)\, | \, 1\leq i \leq t\, \}.$$ 

Also, for $1\leq y \leq p_2$ and ${\bf k'}=(k'_{t+1}, \cdots, k'_{r})$ with $1 \leq \ell_j \leq n_{2j}$, we define 
$$E^t(y, {\bf k'})= \, \{(y+j-1,p_{2j},k'_j)\, | \, t+1\leq j \leq r\, \}.$$

\end{definition}

Note that $O^t(x, {\bf k})$ contains only vertices from odd levels $1, 3, \cdots, 2t-1$ and $E^t(y, {\bf k'})$ contains only vertices from even levels $2(t+1), 2(t+2) \cdots , 2r$.

\begin{lemma}\label{little antichain lemma} 
Let $P=(p_s^{n_s}, \cdots, p_1^{n_1})$ be an $s$-spread and $r=r(P)$. If $A$ is a maximum antichain in $\mathcal{D}_P$, then $A$ has $r$ elements. Furthermore the following hold.
\begin{itemize}
\item[(a)] If $s$ is odd then there exist $x\in \{1, 2, \cdots, p_1\}$ and ${\bf k}=(k_1, \cdots, k_r)$ with $1 \leq k_i \leq n_{2i-1}$ such that $A=O^{r}(x, {\bf k})$.

\item[(b)] If $s$ is even then there exist $t \in \{0,1, \cdots r\}$,  $x\in \{1, 2, \cdots, p_1\}$, $\epsilon \in\{0,1\}$, ${\bf k}=(k_1, \cdots, k_t)$ with $1 \leq k_i \leq n_{2i-1}$, and ${\bf k'}=(k'_{t+1}, \cdots, k'_{r})$ with $1 \leq k'_j \leq n_{2j}$, such that $A=O^t(x, {\bf k}) \cup E^t(x+\epsilon , {\bf k'}).$
\end{itemize}

\end{lemma}
\begin{proof}

First note that since $P$ is a spread, if $1\leq i \leq s$ then $p_{i}=p_1+i-1$.  Thus, by Lemma~\ref{antichain condition}, $O^t(x, {\bf k}) \cup E^t(x+\epsilon , {\bf k'})$ is an antichain, for all $t \in \{0,1, \cdots r\}$,  all $x\in \{1, 2, \cdots, p_1\}$, all $\epsilon \in\{0,1\}$, all ${\bf k}=(k_1, \cdots, k_t)$ with $1 \leq k_i \leq n_{2i-1}$, and all ${\bf k'}=(k'_{t+1}, \cdots, k'_{r(P)})$ with $1 \leq k'_j \leq n_{2j}$. This in particular implies that $O^{r}(x, {\bf k})$ is an antichain in $\mathcal{D}_P$. Since the length of $O^{r}(x, {\bf k})$ is $r=r(P)$, we get $|A|\geq r.$

We write $A=\{(x_i,q_i,k_i)\,|\, 1 \leq i \leq |A|\}$ such that $q_1\leq \cdots \leq q_{|A|}$. By Lemma~\ref{antichain condition}, we must have $q_{i+1}\geq  q_{i}+2$, for $1 \leq i <|A|$. Thus $|A| \leq \lceil s/2 \rceil=r$. Therefore, the length of the maximum antichain $A$ must be exactly $r$.
\bigskip

{\bf Part (a)} Let $s$ be odd. Then $A$ has length $r$ if and only if $q_i=p_{2i-1}$ for all $i$. By Lemma~\ref{antichain condition} we have $x_i<x_{i+1}<q_{i+1}-q_{i-1}+x_i$. Since $P$ is a spread, $q_{i+1}-q_{i}=p_{2i+1}-p_{2i-1}=2$. Thus $x_{i+1}=x_i+1$. Obviously, $1 \leq x_1 \leq p_1$ and $A=O^{r}(x_1,{\bf k})$ for ${\bf k}=(k_1, \cdots, k_{r})$.
\bigskip

{\bf Part (b)} Let $s$ be even. Then $A$ has length $r$ if and only if there exists $t \in\{0, \cdots, r\}$ such that $q_i=p_{2i-1}$ for $ i \leq t $ and $q_i=p_{2i}$ for $t <  i $. Also by Lemma~\ref{antichain condition} $x_i<x_{i+1}<q_{i+1}-q_{i}+x_i$ for all $i$. For $i<t$ and $i \geq t+1$, we have $q_{i+1}-q_{i}=2$ and $q_{t+1}-q_t=3$.
So $x_i=x_1+i-1$ if $i\leq t$. As for $x_{t+1}$, it is either equal to $x_t+1$ or equal to $x_t+2$. Finally if $i>t+1$ then $x_i=x_{t+1}+i-t-1$. 

Therefore $A=O^t(x_1, {\bf k}) \cup E^t(x_1+\epsilon , {\bf k'})$, where ${\bf k}=(k_1, \cdots, k_t)$ and ${\bf k'}=(k_{t+1}, \cdots, k_{r})$ and $\epsilon=x_{t+1}-x_t-1$.

\end{proof}

\begin{proposition}\label{number of parts}
Let $P$ be a partition. All three associated partitions $\lambda(P)$, $\mbox{Q}(P)$ and $\lambda_U(P)$ have  $r(P)$ parts (see Definition~\ref{rP}).
\end{proposition}
\begin{proof}

As mentioned earlier, from \cite[Proposition 2.4]{Basili03} and \cite[Theorem 2.17]{BIK}, we know that $\mbox{Q}(P)$ has $r(P)$ parts. On the other hand, from the definition of a $U$-chain (Definition~\ref{general U chain}) it is clear that $\lambda_U(P)$  must have $r(P)$ parts. So it is enough to show that $\lambda(P)$ has $r(P)$ parts as well.

By Theorem~\ref{conjugate}, the number of parts of $\lambda(P)$ is equal to the maximum cardinality of an antichain in $\mathcal{D}_P$. Write the partition $P$ as $P_\ell \, \cup \cdots \cup \, P_1$ such that each $P_i$ is an $s_i$-spread and the biggest part of $P_i$ is less than or equal to the smallest part of $P_{i+1}$ minus two. One obviously has $r_P=r(P_{\ell})+\cdots+r(P_1)$. 

Let $A$ be an antichain in $\mathcal{D}_P$ and $1\leq i \leq \ell$. Then $A_i=A \cap \mathcal{D}_{P_i}$ is an antichain in $\mathcal{D}_{P_i}$. So by Lemma~\ref{little antichain lemma} $|A_i| \leq r(P_i)$. Thus $$\mid A \mid=\mid A_\ell \mid+\cdots +\mid A_1 \mid\leq r(P_{\ell})+\cdots+r(P_1)=r(P).$$

On the other hand, we can take an appropriate union of maximum antichains in $\mathcal{D}_{P_i}$ to get a maximum antichain in $\mathcal{D}_{P}$. For example consider $$O^{r_1}(1,{\bf k_1}) \cup O^{r_2}(1+r_{1},{\bf k_2}) \cup \cdots \cup O^{r_{\ell}}(1+r_{1}+\cdots+r_{\ell-1},{\bf k_{\ell}})\subset \mathcal{D}_P. $$ By Lemma~\ref{antichain condition} this is an antichain in $\mathcal{D}_P$ and has the desired cardinality $r(P)$.
\end{proof}

Next we prove that $\min(\lambda(P))=\mu(P)$. We will prove this by enumerating disjoint maximum antichains in $\mathcal{D}_P$.  So our goal is to find a set $\mathcal{A}$ consisting of disjoint maximum antichains in $\mathcal{D}_P$ such that $\mathcal{A}$ has the maximum cardinality possible. The following example shows that even if $\mathcal{A}$ is maximal, i.e. it is not strictly contained in any other set of disjoint maximum antichains of $\mathcal{D}_P$, its cardinality may be strictly smaller than the cardinality of another set of disjoint maximum antichains in $\mathcal{D}_P$.

\begin{example}
Let $P=(4^2, 2, 1^2)$. The following sets $\mathcal{A}$ and $\mathcal{B}$ of disjoint maximum antichains are both maximal. 
$$\begin{array}{ll}
\mathcal{A} =&\{\{(1,1,1),(2,4,1)\}, \, \, \{(1,1,2),(2,4,2)\}, \, \, \{(2,2,1),(3,4,1)\}\}. \\
\mathcal{B} =&\{\{(1,1,1),(2,4,1)\}, \, \, \{(1,1,2),(3,4,1)\}, \, \, \{(1,2,1),(2,4,2)\}, \, \, \{(2,2,1),(3,4,2)\}.
\end{array}$$
\end{example}

The following method gives an algorithm to find the maximum possible number of disjoint maximum antichains in $\mathcal{D}_P$. 

\begin{definition}
Let $P=(p_s^{n_s}, \cdots, p_1^{n_1})$ be a partition and $\mathcal{A}$ be a subset of $\mathcal{D}_P$. For $p_1\leq p\leq p_s$ and $1\leq x \leq p$, define $$m_{\mathcal{A}}(x,p)=|\{k\in \mathbb{N}\, |\, (x,p,k)\in A, \mbox{ for some } A\in \mathcal{A}\}|.$$

\end{definition}

In the example above, let $\mathcal{A'} =\{(1,1,1),(2,4,1),(1,1,2),(2,4,2)\}$ be the union of the first two antichain in $\mathcal{A}$ and similarly, let $\mathcal{B'} =\{(1,1,1),(2,4,1), (1,1,2),(3,4,1)\}$, be the union of the first two antichains in $\mathcal{B}$. We have 
$$\begin{array}{l}
 m_{\mathcal{A'}}(2,4)=2, \mbox{ and } m_{\mathcal{A}}(3,4)=0, \mbox{ but } \\
 m_{\mathcal{B'}}(2,4)=1, \mbox{ and } m_{\mathcal{A'}}(3,4)=1.
\end{array}$$

Note that the antichains in $\mathcal{B'}$ are chosen ``more uniformly" than the ones in $\mathcal{A'}$, and this affects the cardinality when they are extended to maximal sets of disjoint antichains, namely $\mathcal{B}$ and $\mathcal{A}$. In fact, in Proposition~\ref{spread min part prop} and Theorem~\ref{spread min part}, we give an algorithm for finding the maximum number of disjoint maximum antichains as ``uniformly" as possible. In each step of our algorithm, we chose maximum antichain so that for each level $p_i$, there exist non-negative integers $m$, $d$ and $d'$ and $d''$ such that the sequence of non-zero multiplicities $m_{\mathcal{A}}(x,p_i)$ has either the form $((m+1)^{d}, m^{d'}, (m+1)^{d''})$ or the form  $(m^{d}, (m+1)^{d'}, m^{d''})$.

The following lemmas are key to our method. They are somewhat technical, so the reader not interested in these details could go directly to the proof of Proposition~\ref{spread min part prop} and see how they are applied. Examples~\ref{theta<etaexample}  and~\ref{eta<thetaexample} also illustrate the Lemmas.

\begin{lemma}\label{strong theta<eta1}
Let $P=((p+g+2)^N,p^M)$ such that $g\geq 1$. Suppose that $\eta, \theta$, $\delta$ and $d$ are integers such that $0 \leq \eta < M$, $1\leq \theta \leq N$ , $0\leq  \delta < p$, and $0\leq d \leq p-\delta$. Assume that $(p+g)\theta\leq p\eta+\delta$. Consider the following subposet of $\mathcal{D}_P$:
\begin{equation}\label{D}
\begin{array}{lll}
\mathcal{D}&=&\{(x,p,k)\,|\, 1\leq x \leq d, 1\leq k\leq \eta\}\\ &\cup&\{(x,p,k)\,|\, d+1\leq x \leq d+\delta, 1\leq k\leq \eta+1\}\\ &\cup&\{(x,p,k)\,|\, d+\delta+1\leq x \leq p, 1\leq k\leq  \eta\}\\ &\cup& \{(y+1,p+g+2,k')\,|\, 1\leq y \leq p+g, 1\leq k'\leq \theta\}.\end{array}
\end{equation}

Write $g\, \theta=pa+b$, for $a\geq 0$ and $0\leq b<p$. 
\begin{itemize}

\item[(a)] If $0\leq b \leq \delta$ and 
$$\begin{array}{ll}
\mathcal{Y}&=\{(y+1,p+g+2,k')\,|\, 1\leq y\leq p+g, 1\leq k'\leq \theta\}, \mbox{ and }\\
\mathcal{X}&=\{(x,p,k)\,|\, x \in\{1, \cdots, d+\delta-b,d+\delta+1, \cdots, p\} \mbox{ and } k \in \{1, \cdots, \theta+a\}\} \\ &\cup \{(x,p,k)\,|\, x \in\{d+\delta-b+1, \cdots, d+\delta\} \mbox{ and } k \in \{1, \cdots, \theta+a+1\}\},
\end{array}$$
then $\mathcal{X}\cup\mathcal{Y}$ is union of $(p+g)\theta$ disjoint antichains in $\mathcal{D}_P$, each of length 2.

\item[(b)] If $\delta <b\leq \delta+d$ and  
$$\begin{array}{ll}
\mathcal{Y}&=\{(y+1,p+g+2,k')\,|\, 1\leq y\leq p+g, 1\leq k'\leq \theta\}, \mbox{ and }\\
\mathcal{X}&=\{(x,p,k)\,|\, x \in\{1, \cdots, b-\delta,d+1, \cdots, d+\delta\} \mbox{ and } k \in \{1, \cdots, \theta+a+1\}\} \\ &\cup \{(x,p,k)\,|\, x \in\{b-\delta+1,\cdots, d, d+\delta+1, \cdots, p\} \mbox{ and } k \in \{1, \cdots, \theta+a\}\},
\end{array}$$
then $\mathcal{X}\cup\mathcal{Y}$ is the union of $(p+g)\theta$ disjoint antichains in $\mathcal{D}_P$, each of length 2.

\item[(c)] If $ \delta+d < b$ and  
$$\begin{array}{ll}
\mathcal{Y}&=\{(y+1,p+g+2,k')\,|\, 1\leq y\leq p+g, 1\leq k'\leq \theta\}, \mbox{ and }\\
\mathcal{X}&=\{(x,p,k)\,|\, x \in\{1, \cdots, d+\delta,p-b+d+\delta+1, \cdots, p\} \mbox{ and } k \in \{1, \cdots, \theta+a+1\}\} \\ &\cup \{(x,p,k)\,|\, x \in\{d+\delta+1,\cdots, p-b+d+\delta\} \mbox{ and } k \in \{1, \cdots, \theta+a\}\},
\end{array}$$
then $\mathcal{X}\cup\mathcal{Y}$ is the union of $(p+g)\theta$ disjoint antichains in $\mathcal{D}_P$, each of length 2. 
\end{itemize}
\end{lemma}

\begin{proof}

By the definition of $\mathcal{D}_P$, in each part of the Lemma the sets $\mathcal{X}$ and $\mathcal{Y}$ are each totally ordered. So we can write $\mathcal{X}=\{X_1, \cdots, X_{(p+g)\theta}\}$ and $\mathcal{Y}=\{Y_1, \cdots, Y_{(p+g)\theta}\}$, such that $X_1<\cdots<X_{(p+g)\theta}$ and $Y_1<\cdots<Y_{(p+g)\theta}$. We prove that $A_c=\{X_c,Y_c\}$ is an antichain for $1\leq c \leq (p+g)\theta$. For $c\in \{1, \cdots, (p+g)\theta\}$, we have $Y_c=(y_c+1, p+g+2, c-\theta(y_c-1))$ with $y=\lceil\frac{c}{\theta}\rceil$, and we write $X_c=(x_c,p,k_c).$ By Lemma~\ref{antichain condition}, to prove that $A_c$ is an antichain, we need to show that 
\begin{equation}\label{x<y<x+g}
x_c\leq y_c\leq x_c+g.
\end{equation}
\medskip

\noindent{\bf Part (a)} First note that since $p\eta+\delta \geq (p+g)\theta=p\theta+pa+b$, we get $p(\eta-\theta-a) \geq b-\delta\geq -\delta$. Since $\delta<p$ and $p(\eta-\theta-a)$ is divisible by $p$, we must have $\eta\geq \theta+a$. So $\mathcal{X}\cup\mathcal{Y}\subseteq \mathcal{D}.$

By definition of $\mathcal{X}$ the following equatities hold for $c\in \{1, \cdots, (p+g)\theta\}$: 
$$c=\left\{\begin{array}{lll}(x_c-1)(\theta+a)+k_c&&\mbox{if } 1\leq x_c \leq d+\delta-b, \\(x_c-1)(\theta+a+1)-d-\delta+b+k_c&&\mbox{if } d+\delta-b+1\leq x_c \leq d+\delta,\\(x_c-1)(\theta+a)+b+k_c&&\mbox{if }d+\delta+1\leq x_c\leq p.\end{array}\right.$$ Also $k_c\in\{1, \cdots, \theta+a+1\}$ if  $d+\delta-b+1\leq x_c \leq d+\delta$ and $k_c\in \{1, \cdots, \theta+a\}$, otherwise.

Thus we have $(x_c-1)\theta \leq (x_c-1)(\theta+a) < c \leq x_c(\theta+a)+b\leq x_c\theta +g\theta.$ So we get $x_c-1<\frac{c}{\theta}\leq x_c+g$, which implies the desired inequalities (\ref{x<y<x+g}) and completes the proof of Part (a).

\medskip

\noindent{\bf Part (b)} Since $p\eta+\delta \geq(p+g)\theta$ and $b-\delta>0$, we get $p(\eta-\theta-a)> 0$. So $\eta\geq \theta+a+1$ and therefore $\mathcal{X}\cup\mathcal{Y}\subseteq \mathcal{D}.$ 

By the definition of $\mathcal{X}$, the following equalities hold for $c\in \{1, \cdots, (p+1)\theta\}$: $$c=\left\{\begin{array}{lll}(x_c-1)(\theta+a+1)+k_c&&\mbox{if } 1\leq x_c \leq b- \delta, \\(x_c-1)(\theta+a)+b-\delta+k_c&&\mbox{if } b-\delta+1\leq x_c \leq d,\\(x_c-1)(\theta+a+1)-d-\delta+b+k_c&&\mbox{if }d+1\leq x_c\leq d+\delta,\\(x_c-1)(\theta+a)+b+k_c&&\mbox{if }d+\delta+1\leq x_c\leq p.\end{array}\right.$$
Also $k_c\in\{1, \cdots, \theta+a+1\}$ if $1 \leq x_c \leq b-\delta$ or $d+1 \leq x_c \leq d+\delta$,  and $k_c\in\{1, \cdots, \theta+a\}$, otherwise. 

Thus $(x_c-1)\theta \leq (x_c-1)(\theta+a) < c \leq x_c(\theta+a)+b\leq x_c\theta +g\theta,$ and we have the desired inequalities (\ref{x<y<x+g}).

\medskip

\noindent{\bf Part (c)} Since $\delta\leq d+\delta<b$, a similar argument to Part (b) shows that  $\eta\geq \theta+a+1$ and thus $\mathcal{X}\cup\mathcal{Y}\subseteq \mathcal{D}.$

The proof is also similar to the proofs of the previous two parts. Here, for $c\in \{1, \cdots, (p+g)\theta\}$ we have the following equalities:
 
$$c=\left\{\begin{array}{lll}(x_c-1)(\theta+a+1)+k_c&&\mbox{if } 1\leq x_c \leq d+\delta, \\(x_c-1)(\theta+a)+d+\delta+k_c&&\mbox{if } d+\delta+1\leq x_c \leq p-b+d+\delta,\\(x_c-1)(\theta+a+1)-p+b+k_c&&\mbox{if }p-b+d+\delta+1\leq x_c\leq p.\end{array}\right.$$ 
Also $k_c\in\{1, \cdots, \theta+a\}$ if $d+\delta+1 \leq x_c \leq p-b+d+\delta$, and  $k_c\in \{1, \cdots, \theta+a+1\}$ otherwise.

\end{proof}

The following example illustrates the application of Lemma~\ref{strong theta<eta1}. 

\begin{example}\label{theta<etaexample}
Let $p=6$, $\eta=4$, $d=2$, $\delta=2$, $\theta=1$, and let $\mathcal{D}$ be defined by Formula (\ref{D}) of Lemma~\ref{strong theta<eta1}. So 
$$\begin{array}{lll}
\mathcal{D}&=&\{(x,6,k)\,|\, 1\leq x \leq 2, 1\leq k\leq 4\}\\ &\cup&\{(x,6,k)\,|\, 3\leq x \leq 4, 1\leq k\leq 5\}\\ &\cup&\{(x,6,k)\,|\, 5\leq x \leq 6, 1\leq k\leq 4\}\\ &\cup& \{(y+1,8+g,1)\,|\, 1\leq y \leq 6+g\}.
\end{array}$$

So $p\eta+\delta=26$ and if $1\leq g\leq 20$ then we will have $(p+g)\theta\leq p\eta+\delta$ and we can use Lemma~\ref{strong theta<eta1} to find $(p+g)\theta=6+g$ disjoint antichains each of length 2. We write $g=6a+b$ such that $0\leq a$ and $0\leq b < 6$.

In Figure 4 we illustrate the vertices of $\mathcal{X}$ as a subset of level $p=6$ of $\mathcal{D}$ for three different values of $g$. In the figures, vertices in level 6 of $\mathcal{D}$ are represented by dots arranged as in Definition~\ref{poset}. The hollowed vertices ($\circ$) in each figure represent the vertices of $\mathcal{D}$ that are in $\mathcal{X}$ for the given value of $g$. 

\begin{itemize}
 \item[(a)] Let $g=7$. Then $a=1$ and $b=1$, and we use the method in part (a) of the Lemma~\ref{strong theta<eta1} to get $\mathcal{X}$. So for all $x\neq 4$, there are exactly two vertices of the form $(x,6,k)$ in $\mathcal{X}$, and there are exactly three vertices of the form $(3,6,k)$ in $\mathcal{X}$. Note that in Lemma~\ref{strong theta<eta1} the vertices of $\mathcal{X}$ are chosen from the bottom of each column, but in Figure 4 we are choosing them from the top of each column in $\mathcal{D}$. This makes the similarity of the pattern of remaining vertices to the original pattern in $\mathcal{D}$ clearer without affecting the validity of the statement of the lemma. 
 
\item[(b)] Let $g=9$. Then $a=1$ and $b=3$, and we use the method in part (b) of the Lemma~\ref{strong theta<eta1} to get $\mathcal{X}$.
\item[(c)] Let $g=11$. Then $a=1$ and $b=5$, and we use the method in part (c) of the Lemma~\ref{strong theta<eta1} to get $\mathcal{X}$.
\end{itemize}

\begin{figure}
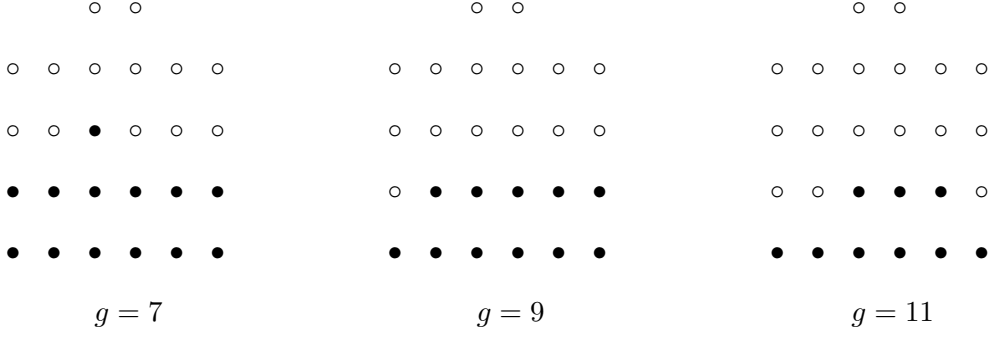
\label{shekl D}
$$\begin{array}{ccccc}
\begin{array}{lllllll}
&&\circ&\circ&&&\\
\circ&\circ&\circ&\circ&\circ&\circ\\
\circ&\circ&\bullet&\circ&\circ&\circ\\
\bullet&\bullet&\bullet&\bullet&\bullet&\bullet\\
\bullet&\bullet&\bullet&\bullet&\bullet&\bullet\\
\end{array}& \hspace{.3in}&
\begin{array}{lllllll}
       &        &\circ&\circ&&&\\
\circ&\circ&\circ&\circ&\circ&\circ\\
\circ&\circ&\circ&\circ&\circ&\circ\\
\circ&\bullet&\bullet&\bullet&\bullet&\bullet\\
\bullet&\bullet&\bullet&\bullet&\bullet&\bullet\\
\end{array}&\hspace{.3in}&
 \begin{array}{lllllll}
       &        &\circ&\circ&       &       &\\
\circ&\circ&\circ&\circ&\circ&\circ\\
\circ&\circ&\circ&\circ&\circ&\circ\\
\circ&\circ&\bullet&\bullet&\bullet&\circ\\
\bullet&\bullet&\bullet&\bullet&\bullet&\bullet\\
\end{array}\\
g=7&&g=9&&g=11
\end{array}$$
\caption{Vertices of $\mathcal{X}$, labeled by $\circ$, as a subset of level 6 of $\mathcal{D}$ for different values of $g$.}
\end{figure}

\end{example}

\begin{lemma}\label{strong eta<theta1}
Let $P=((p+g+2)^N,p^M)$ such that $g\geq 1$. Suppose that $\eta, \theta$, $\epsilon$ and $d$ are integers such that $1\leq \eta \leq M$, $1\leq \theta \leq N$, $0\leq \epsilon <p+g$ and $0\leq d \leq p+g-\epsilon$. Assume that $p\eta\leq (p+g)\theta-\epsilon$. Consider the following subposet of $\mathcal{D}_P$:
\begin{equation}\label{D'}
\begin{array}{lll}
\mathcal{D'}&=&\{(x,p,k)\,|\, 1\leq x\leq p, 1\leq k\leq \eta\} \\ &\cup&\{(y+1,p+g+2,k')\,|\, 1\leq y \leq d, 1\leq k' \leq \theta\}\\&\cup&\{(y+1,p+g+2,k')\,|\, d+1\leq y \leq d+\epsilon, 1\leq k' \leq \theta-1\}\\ &\cup&\{(y+1,p+g+2,k')\,|\, d+\epsilon+1\leq y \leq p+g, 1\leq k'\leq \theta\}.\end{array}
\end{equation}

Write $p\eta=(p+g)a-b$ such that $a\geq 1$ and $0\leq b<p+g$.

\begin{itemize}
\item[(a)] If $0\leq b < \epsilon$ and  
$$\begin{array}{ll}
\mathcal{X}&=\{(x,p,k)\,|\, 1\leq x\leq p, 1\leq k\leq \eta\}, \mbox{ and }\\
\mathcal{Y}&=\{(y+1,p+g+2,k')\,|\, y \in\{1, \cdots, d+\epsilon-b,d+\epsilon+1, \cdots, p+g\} \mbox{ and } k' \in \{1, \cdots, a\}\} \\ &\cup \{(y+1,p+g+2,k')\,|\, y \in\{d+\epsilon-b+1,\cdots, d+\epsilon\} \mbox{ and } k' \in \{1, \cdots, a-1\}\},
\end{array}$$
then $\mathcal{X}\cup\mathcal{Y}$ is union of $p\eta$ disjoint antichains in $\mathcal{D}_P$, each of length 2. 
\item[(b)] If $\epsilon \leq b < d+\epsilon$ and  
$$\begin{array}{ll}
\mathcal{X}&=\{(x,p,k)\,|\, 1\leq x\leq p, 1\leq k\leq \eta\}, \mbox{ and }\\
\mathcal{Y}&=\{(y+1,p+g+2,k')\,|\, y \in\{1, \cdots, b-\epsilon ,d+1, \cdots, d+\epsilon\} \mbox{ and } k' \in \{1, \cdots, a-1\}\} \\ &\cup \{(y+1,p+g+2,k')\,|\, y \in\{b-\epsilon+1,\cdots, d, d+\epsilon+1, \cdots, p+g\} \mbox{ and } k' \in \{1, \cdots, a\}\},
\end{array}$$
then $\mathcal{X}\cup\mathcal{Y}$ is union of $p\eta$ disjoint antichains in $\mathcal{D}_P$, each of length 2. 

\item[(c)] If $d+ \epsilon \leq b$ and 
$$\begin{array}{ll}
\mathcal{X}&=\{(x,p,k)\,|\, 1\leq x\leq p, 1\leq k\leq \eta\}, \mbox{ and }\\
\mathcal{Y}&=\{(y+1,p+g+2,k')\,|\, y \in\{1, \cdots, d+\epsilon, b'+d+\epsilon+1, \cdots, p+g\} \mbox{ and } k' \in \{1, \cdots, a-1\}\} \\ &\cup \{(y+1,p+g+2,k')\,|\, y \in\{d+\epsilon+1,\cdots, b'+d+\epsilon\} \mbox{ and } k' \in \{1, \cdots, a\}\},
\end{array}$$
such that $b'=p+g-b$. Then $\mathcal{X}\cup\mathcal{Y}$ is union of $p\eta$ disjoint antichains in $\mathcal{D}_P$, each of length 2. 
\end{itemize}
\end{lemma}

\begin{proof}

Since by assumption $a\geq 1$ and $b< p+g$, we must have $ga-b>-p$. On the other hand, $ga-b=p(\eta-a)$ is divisible by $p$. Thus $ga-b\geq 0$ and therefore \begin{equation}\label{b/g<a<eta}
0\leq \frac{b}{g} \leq a \leq \eta.
\end{equation} 
We also have 
\begin{equation}\label{g a/eta}
\frac{a}{\eta}=\frac{p}{p+g}+\frac{b}{\eta(p+g)}=1-\frac{g}{p+g}(1-\frac{b}{g\eta}). 
\end{equation}
\medskip

To prove the lemma, in each case we write $\mathcal{X}=\{X_1, \cdots, X_{p\eta}\}$ and $\mathcal{Y}=\{Y_1, \cdots, Y_{p\eta}\}$, such that $X_1<\cdots<X_{p\eta}$ and $Y_1<\cdots<Y_{p\eta}$, and prove that $A_c=\{X_c,Y_c\}$ is an antichain in $\mathcal{D}_P$ for $1\leq c \leq p\eta$. 

For $c\in \{1, \cdots, p\eta\}$. We can write $X_c=(x_c, p, c-\eta(x_c-1))$ with $x_c=\lceil\frac{c}{\eta}\rceil$. We also write $Y_c=(y_c+1,p+g+2,k_c)$. By Lemma~\ref{antichain condition}, to prove that $A_c$ is an antichain, we need to show that 
\begin{equation}\label{y-g<x<y}
y_c-g\leq x_c \leq y_c.
\end{equation}

\noindent{\bf Part (a)}. We know that $(p+g)\theta-\epsilon \geq p\eta= (p+g)a-b$. So $(p+g)(\theta-a)\geq \epsilon-b\geq -b$. Since $b<p+g$, this implies that $\theta \geq a$ and therefore $\mathcal{X}\cup\mathcal{Y}\subseteq \mathcal{D'}.$ 

By the definition of $\mathcal{Y}$, the following equalities hold for $c\in \{1, \cdots, p\eta\}$: $$c=\left\{\begin{array}{lll}(y_c-1)a+k_c&&\mbox{if } 1\leq y_c \leq d+\epsilon -b,\\(y_c-1)(a-1)+d+\epsilon-b+k_c&&\mbox{if } d+\epsilon-b+1\leq y_c\leq d+\epsilon,\\(y_c-1)a-b+k_c&&\mbox{if } d+\epsilon+1\leq y_c\leq p+g.\end{array}\right.$$ 
Also $k_c\in \{1, \cdots, a-1\}$ if $d+\epsilon-b+1\leq y_c \leq d+\epsilon$ and $k_c\in \{1, \cdots, a\}$, otherwise.

So for all $c$, we have $a(y_c-1)-b<c\leq a \, y_c $. Thus $$\frac{a(y_c-1)-b}{\eta}<\frac{c}{\eta}\leq \left(\frac{a}{\eta}\right)\, y_c$$

By inequalities~\ref{b/g<a<eta}, $\left(\frac{a}{\eta}\right)\, y_c\leq y_c$. On the other hand, 
$$\begin{array}{llll}
\left(\frac{a}{\eta}\right)(y_c-1)&=(y_c-1)-g(\frac{y_c-1}{p+g})(1-\frac{b}{g\eta})&&(\mbox{ By Equation }\ref{g a/eta})\\&\geq (y_c-1)-g(1-\frac{b}{g\eta})&&(\mbox{ Since } 0\leq y_c-1<p+g \mbox{ and } b\leq g\eta)\\
&=y_c-g-1+\frac{b}{\eta}.
\end{array}$$

Therefore $y_c-g-1<\frac{c}{\eta}\leq y_c$ and the desired inequalities~\ref{y-g<x<y} follow.

\bigskip

{\bf Part (b)}. Since $(p+g)a-b=p\eta\leq (p+g)\theta-\epsilon$ and $b<p+g$, we get $$-(p+g)<-b\leq \epsilon-b\leq (p+g)(\theta-a).$$ 
So $\theta\geq a$ and therefore $\mathcal{X}\cup\mathcal{Y}\subseteq \mathcal{D'}$.

By the definition of $\mathcal{Y}$ in this case, the following equalities hold for $c\in \{1, \cdots, p\eta\}$: $$c=\left\{\begin{array}{lll}(y_c-1)(a-1)+k_c&&\mbox{if } 1\leq y_c \leq b-\epsilon, \\(y_c-1)a-b+\epsilon+k_c&&\mbox{if } b-\epsilon+1\leq y_c \leq d,\\(y_c-1)(a-1)+d-b+\epsilon+k_c&&\mbox{if } d+1\leq y_c\leq d+\epsilon,\\(y_c-1)a-b+k_c&&\mbox{if } d+\epsilon+1\leq y_c\leq p+g.\end{array}\right.$$ 
Also $k_c\in \{1, \cdots, a-1\}$ if $1\leq y_c \leq b-\epsilon$ or $d+1\leq y_c \leq d+\epsilon$ and $k_c\in \{1, \cdots, a\}$, otherwise.

Now as in the proof of Part (a) above, it is enough to note that we have the inequalities $a(y_c-1)-b<c\leq a \, y_c $ in this case as well.

\bigskip

{\bf Part (c)}. The proof is similar to the proof of the previous two parts. First we note that $(p+g)(\theta-a)\geq b-\epsilon>-(p+g)$. Therefore $\theta\geq a$ and $\mathcal{X}\cup\mathcal{Y}\subseteq \mathcal{D'}$.

Then we use the definition of $\mathcal{Y}$ in this case to get $$c=\left\{\begin{array}{lll}(y_c-1)(a-1)+k_c&&\mbox{if } 1\leq y_c \leq d+\epsilon, \\(y_c-1)a-d-\epsilon+k_c&&\mbox{if } d+\epsilon+1\leq y_c \leq d+\epsilon+b',\\(y-1)(a-1)+b'+k_c&&\mbox{if } d+\epsilon+b'+1\leq y_c\leq p+g.\end{array}\right.$$ We also have $k_c\in \{1, \cdots, a\}$ if $d+\epsilon+1\leq y_c \leq d+\epsilon+b'$ and $k_c\in \{1, \cdots, a-1\}$, otherwise.

It is then easy to see that the inequalities $a(y_c-1)-b<c\leq a \, y_c $ hold in this case too.

\end{proof}

\begin{example}\label{eta<thetaexample}
Let $p=2$, $g=6$, $d=2$, $\epsilon=3$ and $\theta=2$. Assume that $\mathcal{D'}$ is the set defined by Formula~\ref{D'}. 

Using the same notations as in Lemma~\ref{strong eta<theta1}, in Figure 5 we show the vertices of $\mathcal{Y}$ as a subset of level $8$ of $\mathcal{D'}$, for a few different values of $\eta$. In Figure 5, each dot represents a vertex in level 8 of $\mathcal{D'}$, and for each value of $\eta$ the vertices of $\mathcal{Y}$ are hollowed ($\circ$) to distinguish them from other vertices of $\mathcal{D'}$ ($\bullet$). 

\begin{figure}
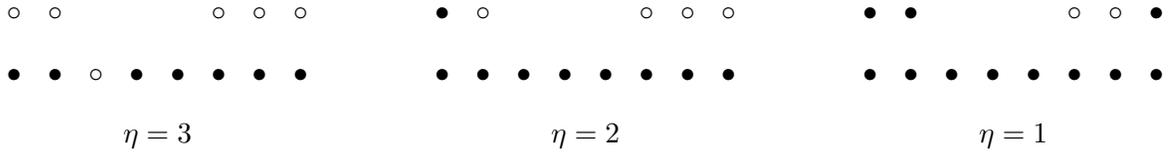
\label{shekl D'}
$$\begin{array}{ccccc}
\begin{array}{lllllllll}
\circ   &\circ   &           &           &           &\circ    &\circ   &\circ\\
\bullet&\bullet&\circ   &\bullet&\bullet&\bullet&\bullet&\bullet
\end{array}
& \hspace{.25in}&
\begin{array}{lllllllll}
\bullet &\circ   &           &           &           &\circ    &\circ   &\circ\\
\bullet&\bullet&\bullet  &\bullet&\bullet&\bullet&\bullet&\bullet
\end{array}
&\hspace{.25in}&
\begin{array}{lllllllll}
\bullet &\bullet&           &           &           &\circ    &\circ   &\bullet\\
\bullet&\bullet&\bullet  &\bullet&\bullet&\bullet&\bullet&\bullet
\end{array}
\\
\eta=3&&\eta=2&&\eta=1
\end{array}$$
\caption{Vertices of $\mathcal{Y}$, labeled by $\circ$, as a subset of level 8 of $\mathcal{D'}$ for different values of $\eta$.}
\end{figure}
\begin{itemize}
 \item[(a)] If $\eta=3$. Then $2\eta=6=8(1)-2$ and $a=1$ and $b=2$. So we use the method in part (a) of  Lemma~\ref{strong eta<theta1} to obtain $\mathcal{Y}$.
 
 \item[(b)] If $\eta=2$. Then $2\eta=4=8(1)-4$ and $a=1$ and $b=4$. Then we use the method in part (b) of  Lemma~\ref{strong eta<theta1} to obtain $\mathcal{Y}$.
 
\item[(c)] If $\eta=1$. Then $2\eta=2=8(1)-6$ and $a=1$ and $b=6$. Then we use the method in part (c) of  Lemma~\ref{strong eta<theta1} to obtain $\mathcal{Y}$.

\end{itemize}
\end{example}

\begin{lemma}\label{strong theta<eta}
Let $P=((p+g+2)^N,p^M)$ such that $g\geq 1$. Suppose that $\eta, \theta$, $\delta$ and $d$ are integers such that $0 \leq \eta < M$, $1\leq \theta \leq N$ , $0\leq  \delta < p$, and $0\leq d \leq p-\delta$. Consider the following subposet of $\mathcal{D}_P$:
\begin{equation}
\begin{array}{lll}
\mathcal{D}&=&\{(x,p,k)\,|\, 1\leq x \leq d, 1\leq k\leq \eta\}\\ &\cup&\{(x,p,k)\,|\, d+1\leq x \leq d+\delta, 1\leq k\leq \eta+1\}\\ &\cup&\{(x,p,k)\,|\, d+\delta+1\leq x \leq p, 1\leq k\leq  \eta\}\\ &\cup& \{(y+1,p+g+2,k')\,|\, 1\leq y \leq p+g, 1\leq k'\leq \theta\}.\end{array}
\end{equation}

Let $\mu=\min\{p\eta+\delta, (p+g)\theta\}$. Then there exists a set $\mathcal{A}$ of disjoint antichains of length 2 in $\mathcal{D}$ such that the following hold for $\bar{\mathcal{A}}=\mathcal{D}\setminus \mathcal{A}$:

\begin{itemize}
\item[(a)] If $(p+g)\theta \leq p\eta+\delta$ then 
\begin{itemize}
\item[(i)] For all $1\leq y \leq p+g$, $m_{\bar{\mathcal{A}}}(y+1,p+3)=0$; and  
\item[(ii)] There exist integers $\eta'$, $\delta'$ and $d'$, satisfying $0\leq \eta'$, $0\leq \delta'< p$ and $0\leq d'\leq p-\delta'$, such that $$m_{\bar{\mathcal{A}}}(x,p)=\left\{\begin{array}{lll} \eta'+1&&\mbox{ if } d'+1\leq x \leq d'+\delta',\\ \eta'&&\mbox{ otherwise}.\end{array} \right.$$
\end{itemize}

\item[(b)] If $p\eta+\delta\leq (p+g)\theta$ then 
\begin{itemize}
\item[(i)] For all $1\leq x \leq p$, $m_{\bar{\mathcal{A}}}(x,p)=0$; and  
\item[(ii)] There exist integers $\theta'$, $\epsilon'$ and $d'$, satisfying $0< \theta'$, $0\leq \epsilon' <p+g$ and $0\leq d' \leq p+q-\epsilon'$, such that $$m_{\bar{\mathcal{A}}}(y,p+3)=\left\{\begin{array}{lll} \theta'-1&&\mbox{ if } d'+1\leq y \leq d'+\epsilon',\\ \theta'&&\mbox{ otherwise}.\end{array} \right.$$
\end{itemize}

\end{itemize}
\end{lemma}

\begin{proof} 

\noindent{\bf Part (a)} This is a consequence of Lemma~\ref{strong theta<eta1} which also gives the following values for $\eta'$, $d'$ and $\delta'$. 
$$\begin{array}{c||c|c|c}
&0\leq b\leq \delta&\delta <b\leq \delta+d&\delta+d<b\\
\hline 
\hline
\eta'&\eta-\theta-a&\eta-\theta-a-1&\eta-\theta-a-1\\
\hline
d'&d&b-\delta&d\\
\hline
\delta'&\delta-b&p-b+\delta&p-b+\delta
\end{array}$$

Here $g\theta=pa+b$ such that $a\geq 0$ and $0\leq b<p$.

\bigskip

\noindent{\bf Part (b)} For $d+1\leq c\leq d+\delta$, let $A_c=\{(c,p,1),(c+1,p+g+2,1)\}$. Then $\{A_c\,|\, d+1\leq c \leq d+\delta\}$ is a set of $\delta$ disjoint antichains of length 2 in $\mathcal{D}$. Let $\mathcal{D}'=\mathcal{D}\setminus \{v\,|\, v\in A_c, \mbox{ for some } d+1 \leq c \leq d+\delta\}.$ So the desired statement is a consequence Lemma~\ref{strong eta<theta1} with $\epsilon=\delta$ and the same $d$. We also get the following values for $\theta'$, $d'$ and $\epsilon'$.

$$\begin{array}{c||c|c|c}
&0\leq b< \delta&\delta \leq b< \delta+d&\delta+d\leq b\\
\hline
\hline
\theta'&\theta-a&\theta-a+1&\theta-a+1\\
\hline
d'&d&b-\delta&d\\
\hline
\epsilon'&\delta-b&p+g-b+\delta&p+g-b+\delta
\end{array}$$

Here $p\eta=(p+g)a-b$ such that $a\geq 1$ and $0\leq b<p+g$.

\end{proof}

\begin{lemma}\label{strong eta<theta}
Let $P=((p+g+2)^N,p^M)$ such that $g\geq 1$. Suppose that $\eta, \theta$, $\epsilon$ and $d$ are integers such that $1\leq \eta \leq M$, $1\leq \theta \leq N$, $0\leq \epsilon <p+g$ and $0\leq d \leq p+g-\epsilon$. Consider the following subposet of $\mathcal{D}_P$:
\begin{equation}
\begin{array}{lll}
\mathcal{D'}&=&\{(x,p,k)\,|\, 1\leq x\leq p, 1\leq k\leq \eta\} \\ &\cup&\{(y+1,p+g+2,k')\,|\, 1\leq y \leq d, 1\leq k' \leq \theta\}\\&\cup&\{(y+1,p+g+2,k')\,|\, d+1\leq y \leq d+\epsilon, 1\leq k' \leq \theta-1\}\\ &\cup&\{(y+1,p+3,k')\,|\, d+\epsilon+1\leq y \leq p+g, 1\leq k'\leq \theta\}.\end{array}
\end{equation}

Let $\mu=\min\{p\eta, (p+g)\theta-\epsilon\}$. Then there exists a set $\mathcal{A}$ of disjoint antichains of length 2 in $\mathcal{D}'$ such that the following hold for $\bar{\mathcal{A}}=\mathcal{D}'\setminus \mathcal{A}$:

\begin{itemize}
\item[(a)] If $p\eta\leq (p+g)\theta-\epsilon$ then 
\begin{itemize}
\item[(i)] For all $1\leq x \leq p$, $m_{\bar{\mathcal{A}}}(x,p)=0$; and  
\item[(ii)] There exist integers $\theta'$ $\epsilon'$ and $d'$, satisfying $0< \theta'$, $0\leq \epsilon' < p+g$ and $0\leq d'\leq p+g-\epsilon'$ such that $$m_{\bar{\mathcal{A}}}(y+1,p+g+2)=\left\{\begin{array}{lll} \theta'-1&&\mbox{ if } d'+1\leq y \leq d'+\epsilon',\\ \theta'&&\mbox{ otherwise}.\end{array} \right.$$
\end{itemize}

\item[(b)] If $(p+g)\theta-\epsilon\leq p\eta$ then 
\begin{itemize}
\item[(i)] For all $1\leq y \leq p+g$, $m_{\bar{\mathcal{A}}}(y+1,p+g+2)=0$; and  
\item[(ii)] There exist integers $\eta'$, $\delta'$ and $d'$, satisfying $0\leq  \eta'$, $0\leq \delta' <  p$, and $0\leq d'\leq p-\delta'$, such that $$m_{\bar{\mathcal{A}}}(x,p)=\left\{\begin{array}{lll} \eta'+1&&\mbox{ if } d'+1\leq x \leq d'+\delta',\\ \eta'&&\mbox{ otherwise}.\end{array} \right.$$
\end{itemize}

\end{itemize}
\end{lemma}

\begin{proof} 

\noindent{\bf Part (a)} This is a consequence of Lemma~\ref{strong eta<theta1} which also gives the following values for $\theta'$, $d'$ and $\epsilon'$. 
$$\begin{array}{c||c|c|c}
&0\leq b< \epsilon&\epsilon \leq b< \epsilon+d&\epsilon+d\leq b\\
\hline
\hline
\theta'&\theta-a&\theta-a+1&\theta-a+1\\
\hline
d'&d&b-\epsilon&d\\
\hline
\epsilon'&\epsilon-b&p+g-b+\epsilon&p+g-b+\epsilon
\end{array}$$

Here $p\eta=(p+g)a-b$ such that $a\geq 1$ and $0\leq b<p+g$.
\noindent{\bf Part (b)} 

\noindent{\bf Case 1.} $\epsilon \geq g$. 

Then for $c\in\{1, \cdots, d\}$, let $A_c=\{(c-1,p,1),(c,p+g+2,1)\}$ and for $c\in\{d+\epsilon+1, \cdots, p+g\}$, let $A_c=\{(c-g-1,p,1),(c,p+g+2,1)\}$. Then $\{A_c\,|\, c\in \{1, \cdots, d, \epsilon+d+1, \cdots, p+g\}\}$ is a set of $(p+g-\epsilon)$ disjoint antichains of length 2 in $\mathcal{D}'$. Let $\mathcal{D}=\mathcal{D}'\setminus \{v\,|\, v\in A_c, \mbox{ for some } c\in \{1, \cdots, d, \epsilon+d+1, \cdots, p+g\}\}.$ So the desired statement of Part (b) is a consequence of Lemma~\ref{strong theta<eta1} with $\delta=\epsilon-g$ and the same $d$. We also get the following values for $\eta'$, $d'$ and $\delta'$.

$$\begin{array}{c||c|c|c}
&0\leq b+g\leq \epsilon&\epsilon <b+g\leq \epsilon+d&\epsilon+d<b+g\\
\hline 
\hline
\eta'&\eta-\theta-a&\eta-\theta-a-1&\eta-\theta-a-1\\
\hline
d'&d&b-\epsilon-g&d\\
\hline
\delta'&\epsilon-g-b&p-b+\epsilon-g&p-b+\epsilon-g
\end{array}$$

Here $g\theta=pa+b$ such that $a\geq 0$ and $0\leq b<p$.

\noindent{\bf Case 2.} $\epsilon<g$. 

$\bullet$ If $d> g-\epsilon $, then $|\{d+\epsilon+1, \cdots, p+g\}|< p$, and for any $c\in \{d+\epsilon+1, \cdots, p+g\}$, $A_c=\{(c-g, p, 1),(c+1, p+g+2, 1)\}$ is an antichain.  Since $\epsilon <g$, for $c\in\{g-\epsilon+1, \cdots, d\}$, $A_c=\{(c+\epsilon-g, p, 1),(c+1, p+g+2,1)\}$ is a antichain, as well. Finally, for $c\in\{1, \cdots, g-\epsilon\}$, we match the vertices $(c+1, p+g+2,1)$ with vertices in level $p$, from left to right. To do this, we write $c=pk_c+\ell_c$ such that $k_c \geq 0$ and $1\leq  \ell_c \leq p$. Then 
\begin{equation}
\ell_c \leq c \leq g-\epsilon \leq \ell_c+g.
\end{equation}
 So 
 \begin{equation}\label{left-right match}
 A_c=\{(\ell_c, p, k_c+1),(c+1, p+g+2,1)\}
\end{equation}
is an antichain for any $c\in \{1, \cdots, g-\epsilon\}$. 
Now if we set $$\mathcal{D}=\mathcal{D}'\setminus\{v\,|\, v\in A_c \mbox{ for some } c\in \{1, \cdots, d, d+\epsilon+1, \cdots, p+g\}\}, $$ then the desired statement of Part (b) is a consequence of Part (a) of Lemma~\ref{strong theta<eta} above applied to $\mathcal{D}$. Here $\eta$ is $\eta-k-1$, $d=\ell$ and $\delta=p-\ell$, where $g-\epsilon=pk+\ell$ such that $k \geq 0$ and $1\leq  \ell \leq p$.

$\bullet$ If $d \leq  g-\epsilon $, then for all $c\in\{1, \cdots, d\}$, we match the vertices $(c+1, p+g+2, 1)$ with vertices in level $p$, from left to right as we showed in Equation~\ref{left-right match} above. Suppose that $d=pk+\ell$, such that $k \geq 0$ and $1\leq  \ell \leq p$. Then $A_d=\{(\ell, p, k+1),(c+1, p+g+2,1)\}$. Since $\ell \leq d\leq g-\epsilon$, we have $\ell \leq d \leq d+\epsilon \leq g \leq \ell+g.$ 

So for each $c\in\{d+\epsilon+1, \cdots, d+\epsilon+p-\ell\}$, the set $A_c=\{(\ell+c-d-\epsilon,p, k+1),(c+1, p+g+2, 1)\}$ is an antichain. 

Finally, we will match vertices $(c+1, p+g+2, 1)$ for all $c\in\{d+\epsilon+p-\ell+1, \cdots, p+g\}$ with vertices in the level $p$, from right to left. Note that there are $g-d-\epsilon+\ell\leq g$ such $c$'s, and starting from right the matching will be similar (dual) to the one explained in Equation~\ref{left-right match} above. Let $\{A_c\,|\, c\in \{d+\epsilon+p-b+1, \cdots, p+g\}\}$ be the set of such antichains and define $$\mathcal{D}=\mathcal{D}'\setminus\{v\,|\, v\in A_c \mbox{ for some } c\in \{1, \cdots, d, d+\epsilon+1, \cdots, p+g\}\}. $$ The proof is then complete once Part (a) of Lemma~\ref{strong theta<eta} above is applied to $\mathcal{D}$ with $\eta$ equal to $\eta-k'-2$, $d=0$ and $\delta=\ell'-1$, where $g-\epsilon=pk'+\ell'$ such that $k' \geq 0$ and $1\leq  \ell' \leq p$.

\end{proof}

Now we have the necessary tools to prove the following proposition.

\begin{proposition} \label{spread min part prop}
Let the partition $P$ of $n$ be a spread. Then $\min(\lambda(P))=\mu(P).$
\end{proposition}

\begin{proof}

Since $\lambda(P) \geq \lambda_U(P)$, and $\lambda(P)$ and $\lambda_U(P)$ have the same number of parts (by Proposition~\ref{number of parts}), we get $\min(\lambda(P))\leq \min(\lambda_U(P))$. So by Theorem~\ref{Oblak min}, \begin{equation}\min(\lambda(P)) \leq \mu(P).\end{equation} 

Using Theorem~\ref{Greene}, we prove the reverse inequality by finding $\mu(P)$ disjoint maximum antichains in $\mathcal{D}_P$. 

Assume that $P$ is an $s$-spread and write $P=(p_s^{n_s}, \cdots, p_1^{n_1})$, with $p_i=p_1+i-1$ and $n_i>0$, for $1 \leq i \leq s$.  Let $r=r(P).$
\bigskip

\noindent {\bf Case 1.} $s$ is odd.

Then $r=\displaystyle{\frac{s+1}{2}}$ and $\mu(P)=p_1 \cdot \min \{n_{2i-1}\, | \, i=1, \cdots, r\}.$ 

For a positive integer $j$, satisfying $1 \leq j \leq \min \{n_{2i-1}\, | \, i=1, \cdots, r\},$ let ${\bf j}=(j, \cdots, j)$. Recall from Definition~\ref{O and E} that for $1 \leq x \leq p_1$, $$O^{r}(x, {\bf j}) = \, \{(x+i-1,p_{2i-1},j)\, | \, 1\leq i \leq r\, \}.$$

Note that $O^{r}(x, {\bf j}) \cap O^{r}(y, {\bf k}) =\emptyset \mbox{  if } x \neq y \mbox{ or } j \neq k.$ 

Varying $x$ and $j$, we get $\mu(P)$ disjoint maximum antichains (each of length $r$) in $\mathcal{D}_P$, as desired.

\noindent {\bf Case 2.} $s$ is even.

Then $r=\displaystyle{\frac{s}{2}}$ and $\mu(P)=\min \{p_1 n_{2i-1}+(p_1+1) n_{2j} \, | \, 1 \leq i \leq j \leq r \}.$  Let $p=p_1$. First note that by Lemma~\ref{little antichain lemma}, a vertex $(x,p_{2i-1},k)$ can belong to a maximum antichain only if $i\leq x\leq p+i-1$, and a vertex $(y,p_{2j},k)$ can belong to a maximum antichain only if $j\leq y \leq p+j$. So we are going to restrict our study to the subposet $\mathcal{D}$ of $\mathcal{D}_P$ that only includes such vertices. So  $\mathcal{D}=\left \{(x,p_i,k)\in \mathcal{D}_P\,|\, \lceil \frac{i}{2} \rceil \leq x \leq p+\lfloor \frac{i}{2}\rfloor\right \}.$

To complete the proof, we give an algorithm to construct a set $\mathcal{A}$ of disjoint maximum antichains satisfying the following two properties for some integers $1\leq \frak{b} \leq \frak{t} \leq r$. 

\begin{itemize}
\item[({\bf i)}] If $\bar{\mathcal{A}}=\mathcal{D} \setminus \mathcal{A}$ then $m_{\bar{\mathcal{A}}}(x,p_{2\frak{b}-1})=0$ for all $x\in\{\frak{b},\cdots, p+\frak{b}-1\}$, and  $m_{\bar{\mathcal{A}}}(y,p_{2\frak{t}})=0$ for all $y\in\{\frak{t}, \cdots,  p+\frak{t}\}$, {\it i.e.} every element of $\mathcal{D}$ in level $2\frak{b}-1$ or $2\frak{t}$ belongs to an antichain in $\mathcal{A}$; and 
\item[({\bf ii)}] For all $A\in \mathcal{A}$ if $m_A(x,p_{2\frak{b}-1})>0$ for some $x\in \{\frak{b},\cdots, p+\frak{b}-1\}$ then $m_A(y,p_{2\frak{t}})=0$ for all  $y\in \{\frak{t},\cdots, p+\frak{t}\}$, {\it i.e.} no antichain in $\mathcal{A}$ contains vertices from both levels $2\frak{b}-1$ and $2\frak{t}$.
\end{itemize}

Set 
 $$\begin{array}{lll}
\frak{o}&:=&\min\{n_{2i-1}\,|\, 1\leq i \leq r\},\\
\frak{e}&:=&\min\{n_{2j}\,|\, 1\leq j \leq r\}, \mbox{ and}\\
 \mathcal{A}&:=&\{O^{r}(x, {\bf k})\,|\, 1 \leq x \leq p, 1 \leq k \leq \frak{o} \mbox{ and } {\bf k}=(k, \cdots, k)\}\\ &\cup& \{E^{0}(y, {\bf k})\,|\, 1 \leq y \leq p+1, 1 \leq k \leq \frak{e} \mbox{ and } {\bf k}=(k, \cdots, k)\}.\\
 \end{array}$$

Then $\mathcal{A}$ satisfies the two conditions (i) and (ii) above if and only if there exists $\frak{b}$ and  $\frak{t}$ such that $\frak{b}\leq \frak{t}$, $n_{2\frak{b}-1}=\frak{o}$ and $n_{2\frak{t}}=\frak{e}$. So the proof is complete if this is the case. Otherwise, set $\mathcal{D}:=\mathcal{D}\setminus\mathcal{A}$ and $$\frak{t}:=\max\{j\,|\, 1\leq j \leq r, \, \, m_{\mathcal{D}}(y,p_{2j})=0, \mbox{ for all } y\}.$$

In other words, we choose $\frak{t}$ so that the level $2\frak{t}$ is the highest even level which is empty in $\mathcal{D}$. By assumption $m_{\mathcal{D}_1}(x,p_{2i-1})>0$ for all $i\in \{1, \cdots, \frak{t}\}$. By the choice of $\frak{t}$ we have $m_{\mathcal{D}}(y,p_{2j})>0$ for all $j\in \{\frak{t}+1, \cdots, r\}$. Note that, by the definition of $\mathcal{A}$,  for each $\ell \in\{1, \cdots, s\}$, $m_{\mathcal{D}}(x,p_\ell)$ is the same for all $x$. We are now going to extend $\mathcal{A}$ by adding antichains of the form $O^{\frak{t}}(x,{\bf k})\cup E^{\frak{t}}(y,{\bf k'})$.
Set 
$$\begin{array}{l}
\frak{o}:=\min\{m_{\mathcal{D}}(x,p_{2i-1})\,|\, 1\leq i \leq \frak{t}\}, \mbox{ and }\\
\frak{e}:=\min\{m_{\mathcal{D}}(y,p_{2j})\,|\, \frak{t}+1\leq j \leq r\}.
\end{array}$$ 

By Lemma~\ref{antichain condition}, the problem of finding antichains of form $O^{\frak{t}}(x,{\bf k})\cup E^{\frak{t}}(y,{\bf k'})$ in $\mathcal{D}$ can be reduced to the problem of finding antichains of length 2 in $\mathcal{D}_{Q_1}$ for $Q_1=((p+3)^{\frak{e}},p^{\frak{o}})$.

{\bf Case 2-1.} $p\frak{o} \leq (p+1)\frak{e}$. 

Then we use Lemma~\ref{strong eta<theta} to find $p\frak{o}$ disjoint antichains of the form $O^{\frak{t}}(x,{\bf k})\cup E^{\frak{t}}(y,{\bf k'})$ in $\mathcal{D}$ and extend $\mathcal{A}$ by adding these antichains. Let $\frak{b}\in \{1, \cdots, \frak{t}\}$ be such that $m_{\mathcal{D}}(x,p_{2\frak{b}-1})=\frak{o}$. Then $\mathcal{A}$ satisfies Conditions (i) and (ii) above the proof is complete in this case.

{\bf Case 2-2.} $p\frak{o}>(p+1)\frak{e}$. 

Then we use Lemma~\ref{strong theta<eta}, to find $(p+1)\frak{e}$ disjoint antichains of the form $O^{\frak{t}}(x,{\bf k})\cup E^{\frak{t}}(y,{\bf k'})$ in $\mathcal{D}$ and extend $\mathcal{A}$ by adding these antichains. Now we set $\mathcal{D}=\mathcal{D}\setminus\mathcal{A}.$

By Lemma~\ref{strong theta<eta}, for each $i\in \{1, \cdots, \frak{t}\}$, there exist integers $\eta(i)\geq 0$  and such that $0\leq \delta < p$ such that $m_{\mathcal{D}_2}(x,p_{2i-1})=\left\{\begin{array}{lll} \eta(i)+1&&\mbox{ if } i\leq x\leq i+\delta-1\\ \eta(i)&&\mbox{ if } i+\delta \leq x\leq i+p-1.\end{array}\right.$. Also for all $j\in\{\frak{t}+1, \cdots, r\}$, each $m_{\mathcal{D}}(y,p_{2j})$ decreases by $\frak{e}$. 

Now we replace $t$ by a strictly larger value as follows.
$$\begin{array}{l}
\frak{t}:=\max\{j\,|\, 1\leq j \leq r, \, m_{\mathcal{D}}(y,p_{2j})=0, \mbox{ for all } y\},\\
\frak{o}:=\min\{\eta(i)\,|\, 1\leq i \leq \frak{t}\}, \mbox{ and }\\
\frak{e}:=\min\{m_{\mathcal{D}}(y,p_{2j})\,|\, \frak{t}+1\leq j \leq r\}.\\
\end{array}$$

{\bf Case 2-2-1.} $p\frak{o}+\delta \leq (p+1)\frak{e}$. 

Then we use Lemma~\ref{strong eta<theta} to extend $\mathcal{A}$ by adding $(p\frak{o}+\delta)$ disjoint antichains of type $O^{\frak{t}}(x,{\bf k})\cup E^{\frak{t}}(y,{\bf k'})$ in $\mathcal{D}$. Let $\frak{b}$ is an integer in $\{1, \cdots, \frak{t}\}$ such that $\eta(i)=\frak{o}$. Then $\mathcal{A}$ satisfies both conditions (i) and (ii) and the proof is complete in this case.
\medskip

{\bf Case 2-2-2.} $p\frak{o}+\delta > (p+1)\frak{e}$. 

Then we use Lemma~\ref{strong eta<theta} to extend $\mathcal{A}$ by adding $(p+1)\frak{e}$ disjoint antichains of type $O^{\frak{t}}(x,{\bf k})\cup E^{\frak{t}}(y,{\bf k'})$ in $\mathcal{D}$. We then set $\mathcal{D}:=\mathcal{D}\setminus\mathcal{A}$ and repeat as in Case 2-2. Since the value of $\frak{t}$ is strictly increasing by repeating the process, it will eventually terminate by producing the desired set $\mathcal{A}$.

\end{proof}

We summarize the proof of Proposition~\ref{spread min part prop} in the following algorithm.

\begin{algorithm} 
Let $P=(p_s^{n_s}, \cdots, p_1^{n_1})$ be a spread with $p_i=p_1+i-1$ and $n_i>0$ for $1\leq i \leq s$. Set $$\begin{array}{ll}
 r=&r(P)=\lceil \frac{s}{2}\rceil;  \, \, p=p_1;\\
 \mathcal{D}=&\{(x,p_i,k)\,|\, 1\leq i \leq s, 1 \leq k \leq n_i, \lceil \frac{i}{2} \rceil \leq x \leq p+\lfloor \frac{i}{2}\rfloor\};\\
\frak{o}=&\min\{n_{2i-1}\,|\, 1\leq i \leq r\}; \, \, \frak{e}_0=\min\{n_{2j}\,|\, 1\leq j \leq r\}; \mbox{ and}\\
 \mathcal{A}=&\{O^{r}(x, {\bf k})\,|\, 1 \leq x \leq p, 1 \leq k \leq \frak{o} \mbox{ and } {\bf k}=(k, \cdots, k)\} \cup \\  &\{E^{0}(y, {\bf k})\,|\, 1 \leq y \leq p+1, 1 \leq k \leq \frak{e}_0 \mbox{ and } {\bf k}=(k, \cdots, k)\}.\\
 \end{array}$$

\noindent{\bf Input:} An $s$-spread $P=(p_s^{n_s}, \cdots, p_1^{n_1})$.

\noindent{\bf Output:} A set $\mathcal{A}$ consisting of $\mu(P)$ disjoint maximum antichains in $\mathcal{D}_P$.
$$\begin{array}{ll} 
\bullet& \mathcal{D}:=\mathcal{D}\setminus\mathcal{A}.\\
\bullet& \frak{t}:=\max\{j\,|\, 1\leq j \leq r \mbox{ and }m_{\mathcal{D}}(y,p_{2j})=0 \mbox{ for all } y\}.\\
\bullet &\mbox{For } 1\leq i \leq \frak{t}, \mbox{ let } \eta(i)\geq 0 \mbox{ and } 0\leq \delta < p \mbox{ be such that }\\
& m_{\mathcal{D}}(x,p_{2i-1})=\left\{\begin{array}{lll} \eta(i)+1&&\mbox{ if } i\leq x\leq i+\delta-1,\\ \eta(i)&&\mbox{ if } i+\delta \leq x\leq i+p-1.\end{array}\right.\\
\bullet &\frak{o}:=\min\{\eta(i)\,|\, 1\leq i \leq \frak{t}\}.\\
\bullet & \frak{e}:=\min\{m_{\mathcal{D}}(y,p_{2j})\,|\, \frak{t}+1\leq j \leq r\}.\\
&\begin{array}{ll}
\bullet & \mbox{ If } p\frak{o}+\delta=0 \mbox{ then {\bf End.}}\\
\bullet & \mbox{ If } 0<p\frak{o}+\delta \leq (p+1)\frak{e} \mbox{ then}\\
	   &\mathcal{A}:=\mathcal{A}\cup\{\mbox{all } O^{\frak{t}}(x,{\bf k})\cup E^{\frak{t}}(y,{\bf k'}) \mbox{ obtained by Lemma~\ref{strong eta<theta}}.\}. \mbox{ {\bf End.}}\\
\bullet & \mbox{ If } p\frak{o}+\delta > (p+1)\frak{e}  \mbox{ then {\bf repeat} with}\\
	   &\mathcal{A}:=\mathcal{A}\cup\{\mbox{all } O^{\frak{t}}(x,{\bf k})\cup E^{\frak{t}}(y,{\bf k'}) \mbox{ obtained by Lemma~\ref{strong theta<eta}}.\}.
\end{array}	  
\end{array}$$
\end{algorithm}

\begin{remark}

Note that there are alternative algorithms to the one discussed in the proof above. In the proof we used Lemmas~\ref{strong eta<theta} and~\ref{strong theta<eta} with $d=0$, for simplicity.

We can also follow an alternative algorithm taking antichains of type $\frak{t}-1$, where $\frak{t}$ is the lowest odd level that becomes empty after each step. 
\end{remark}

The following corollary summarizes some of the nice properties of the set of antichains constructed in the proof of Proposition~\ref{spread min part prop}.

Recall from Equation (\ref{mu  spread equi}) in Definition~\ref{mu defn}, for a spread $P$ with smallest part $p$, $\mu(P)=\min\{pn_{2i-1}+(p+1)n_{2j}\,|\, 1\leq i \leq j \leq r(P)\}.$  
\begin{corollary}\label{chiz}
Let $P$ be a spread with smallest part $p$ and biggest part $q$ and $r=r(P)$. Suppose that $\mu(P)=pm_1+(p+1)m_2$. Then we can find a set $\mathcal{A}$ of $\mu(P)$ disjoint maximum antichains such that  
 
\begin{itemize}
\item[(a)] There exist non-negative integers $a$ and $\delta$ such that $\delta<p$, and $$m_{\mathcal{A}}(x,p)=\left\{\begin{array}{lll}a&&\mbox{ if } 1\leq x \leq \delta,\\a+1&&\mbox{ if } \delta+1\leq x \leq p.\end{array}\right.$$
\item[(b)] There exists a non-negative integer $b$ such that $m_{\mathcal{A}}(y,p+1)=b$, for $1\leq y\in\leq p+1$.
\item[(c)] There exists a non-negative integer $c$ such that $m_{\mathcal{A}}(x,q-1)=c$, for all $\lceil\frac{q-p}{2}\rceil \leq x\leq r+p-1$.
\item[(d)] There exist non-negative integers $d$ and $\epsilon$ such that $\epsilon \leq p$, and $$m_{\mathcal{A}}(y,q)=\left\{\begin{array}{lll}d+1&&\mbox{ if } r\leq y < r+\epsilon,\\d&&\mbox{ if } r+\epsilon \leq y \leq \lceil\frac{q-p}{2}\rceil+p.\end{array}\right.$$

\end{itemize}

\end{corollary}

By symmetry of $\mathcal{D}_P$, we also have the following corollary.

\begin{corollary}\label{chiz prime}
Let $P$ be a spread with smallest part $p$ and biggest part $q$ and $r=r(P)$. Suppose that $\mu(P)=pm_1+(p+1)m_2$. Then we can find a set $\mathcal{A}$ of $\mu(P)$ disjoint maximum antichains such that  

\begin{itemize}
\item[(a)] There exist non-negative integers $a$ and $\delta$ such that $\delta<p$, and $$m_{\mathcal{A}}(x,p)=\left\{\begin{array}{lll}a+1&&\mbox{ if } 1\leq x \leq \delta,\\ a&&\mbox{ if } \delta+1\leq x \leq p.\end{array}\right.$$
\item[(b)] There exists a non-negative integer $b$ such that $m_{\mathcal{A}}(y,p+1)=b$, for $1\leq y\in\leq p+1$.
\item[(c)] There exists a non-negative integer $c$ such that $m_{\mathcal{A}}(x,q-1)=c$, for all $\lceil\frac{q-p}{2}\rceil \leq x\leq r+p-1$.
\item[(d)] There exist non-negative integers $d$ and $\epsilon$ such that $\epsilon \leq p$, and $$m_{\mathcal{A}}(y,q)=\left\{\begin{array}{lll}d&&\mbox{ if } r\leq y < r+\epsilon,\\d+1&&\mbox{ if } r+\epsilon \leq y \leq \lceil\frac{q-p}{2}\rceil+p.\end{array}\right.$$

\end{itemize}
\end{corollary}

We now generalize Proposition~\ref{spread min part prop} to an arbitrary partition $P$.

\begin{theorem} \label{spread min part}
For any partition $P$ of $n$,  $\min(\lambda(P))=\mu(P).$
\end{theorem}

\begin{proof}

As mentioned in the proof of Proposition~\ref{spread min part prop} above, to prove the theorem it is enough to  find $\mu(P)$ disjoint maximum antichains in $\mathcal{D}_P$.

Write $P=P_\ell \, \cup \cdots \cup \, P_1$ such that each $P_k$ is a spread and the biggest part of $P_k$ is less than or equal to the smallest part of $P_{k+1}$ minus two, for all $1\leq k<\ell$. We prove the claim by induction on $\ell$.

If $\ell=1$, then $P$ is a spread and the equality is proved in Proposition~\ref{spread min part prop} above. 

Now suppose that $\ell>1$ and assume that the desired equality holds for any union of $\ell-1$ or fewer spreads. 

Consider the partition $\bar{Q}=(P_{\ell}, \cdots, P_2)$ and let $\mathcal{D}_Q$ be the subposet of $\mathcal{D}_{\bar{Q}}$ obtained by removing $r=r(P_1)$ vertices from each end of each row of $\mathcal{D}_{\bar{Q}}$ (so the corresponding partition $Q=\bar{Q}-2r(P_1)$). By the inductive hypothesis, there are $\mu(P_1)$ disjoint antichains, each of length $r$, in $\mathcal{D}_{P_1}$. There are also $\mu(Q)$ disjoint antichains, each of length $r(Q)=r(\bar{Q})$, in $\mathcal{D}_{Q}$. 

By Definition~\ref{mu defn}, $\mu(P)=\min\{\mu(P_1),\mu(Q)\}$. Every maximum antichain in $\mathcal{D}_P$ is a union of a maximum antichain in $\mathcal{D}_{P_1}$ and an antichain corresponding to a maximum antichain in $\mathcal{D}_Q$. By Lemma~\ref{antichain condition} and Corollaries~\ref{chiz} and~\ref{chiz prime}, to complete the proof, it is enough to prove the following claim:

\noindent{\bf Claim.} Let $P=((q+3)^{D},(q+2)^{C},p^{B},(p-1)^{A})$ with $q\geq p$. Assume that $a,b,c,d, \epsilon$ and $\delta$ are non-negative integers such that $a\leq A$, $b\leq B$, $c\leq C$, $d\leq D$, $\delta<q$, and $\epsilon < p$. We define the subposet  $\mathcal{E}$ of $\mathcal{D}_P$ as follows:
$$\begin{array}{lll}
\mathcal{E}&=&\mathcal{A}\cup\mathcal{B}\cup\mathcal{C}\cup\mathcal{D}, \mbox{ such that}\\
\mathcal{A}&=&\{(x,p-1,k)\,|\, 1\leq x\leq p-1, 1\leq k \leq a\}\\
\mathcal{B}&=&\{(x,p,k)\,|\, 1\leq x\leq \epsilon, 1\leq k \leq b+1\} \cup \{(x,p,k)\,|\, \epsilon< x\leq p, 1\leq k \leq b\}\\
\mathcal{C}&=&\{(y+1,q,k)\,|\, 1\leq y\leq \delta, 1\leq k \leq c+1\} \cup \{(y+1,q,k)\,|\, \delta< y\leq q, 1\leq k \leq c\}\\
\mathcal{D}&=&\{(y+1,q+1,k)\,|\, 1\leq y\leq q+1, 1\leq k \leq d\}.
\end{array}$$  
Then there are $\mu(P)$ disjoint antichains of length 2 in $\mathcal{E}$. 
\bigskip

We begin by matching vertices $\{(x,p,1)\,|\, 1\leq x\leq \epsilon\}$ from $\mathcal{B}$ with vertices $\{(y+1,q,k)\,|\, 1\leq y\leq \delta\}$ from $\mathcal{C}$, and removing those antichains. There will be $\min\{\epsilon, \delta\}$ such antichains. Then we remove all of the vertices of such antichains to obtain subposets $\mathcal{B}_1\subseteq \mathcal{B}$ and $\mathcal{C}_1\subseteq \mathcal{C}$. 

Next we match vertices from $\mathcal{B}_1$ and $\mathcal{C}_1$ to obtain more antichains of length 2. If $q=p$, then there is only a trivial matching of these vertices. If $q>p$, then we can apply the matching methods of Lemma~\ref{strong theta<eta} (if $\epsilon \leq \delta$) or Lemma~\ref{strong eta<theta} (if $\epsilon \geq \delta$) to obtain disjoint length 2 antichains in $\mathcal{B}_1\cup\mathcal{C}_1$. Now remove all the vertices belonging to any of such antichains to get subposets $\mathcal{B}_2\subseteq \mathcal{B}_1$ and $\mathcal{C}_2\subseteq \mathcal{C}_1$. Then $\mathcal{B}_2=\emptyset$ or $\mathcal{C}_2=\emptyset$. 
\bigskip

If $\mathcal{B}_2=\emptyset$, apply Lemma~\ref{strong theta<eta} or Lemma~\ref{strong eta<theta} to $\mathcal{C}_2\cup\mathcal{A}$ to obtain more disjoint antichains of length 2. If this uses up all vertices in $\mathcal{A}$, then the proof is complete. Otherwise, to complete the proof, we apply the Lemmas one more time, this time to the remaining vertices $\mathcal{A}_1\cup\mathcal{D}$. 
\bigskip

If $\mathcal{C}_2=\emptyset$, we apply Lemma~\ref{strong theta<eta} or Lemma~\ref{strong eta<theta} to $\mathcal{B}_2\cup\mathcal{D}$. If this uses up all vertices in $\mathcal{D}$, then the proof is complete. Otherwise, we will need to apply the Lemmas one more time, this time to the remaining vertices $\mathcal{A}\cup\mathcal{D}_1$ to complete the proof. 

\end{proof}

\begin{example}\label{min 10} Let $P_1=(6^2,5,4,3^2)$. Then 
$$\begin{array}{l}
\mu(P_1)=\min\{3(2)+4(1),3(2)+4(2),3(1)+4(2)\}=10.
\end{array}$$ 
Using a method similar to the one discussed in the proof of Proposition~\ref{spread min part prop}, we can get the following 10 disjoint maximum antichains (length-2) in $\mathcal{D}_{P_1}$ (see Figure 6).
$$\begin{array}{l}
O_1=\{(1,3,1),(2,5,1)\},
\,\, O_2=\{(2,3,1),(3,5,1)\},\,\, 
O_3=\{(3,3,1),(4,5,1)\},\\

E_1=\{(1,4,1),(2,6,1)\},
\,\, E_1=\{(2,4,1),(3,6,1)\},\,\, 
E_3=\{(3,4,1),(4,6,1)\},\,\, E_4=\{(4,4,1),(5,6,1)\},\\

T_1=\{(1,3,2),(2,6,2)\},
\,\, T_2=\{(2,3,2),(3,6,2)\},\,\, 
T_3=\{(3,3,2),(4,6,2)\}.

\end{array}$$

\begin{figure}
\vspace{-.5 in}
\begin{center}\label{example1}
\includegraphics[scale=.35]{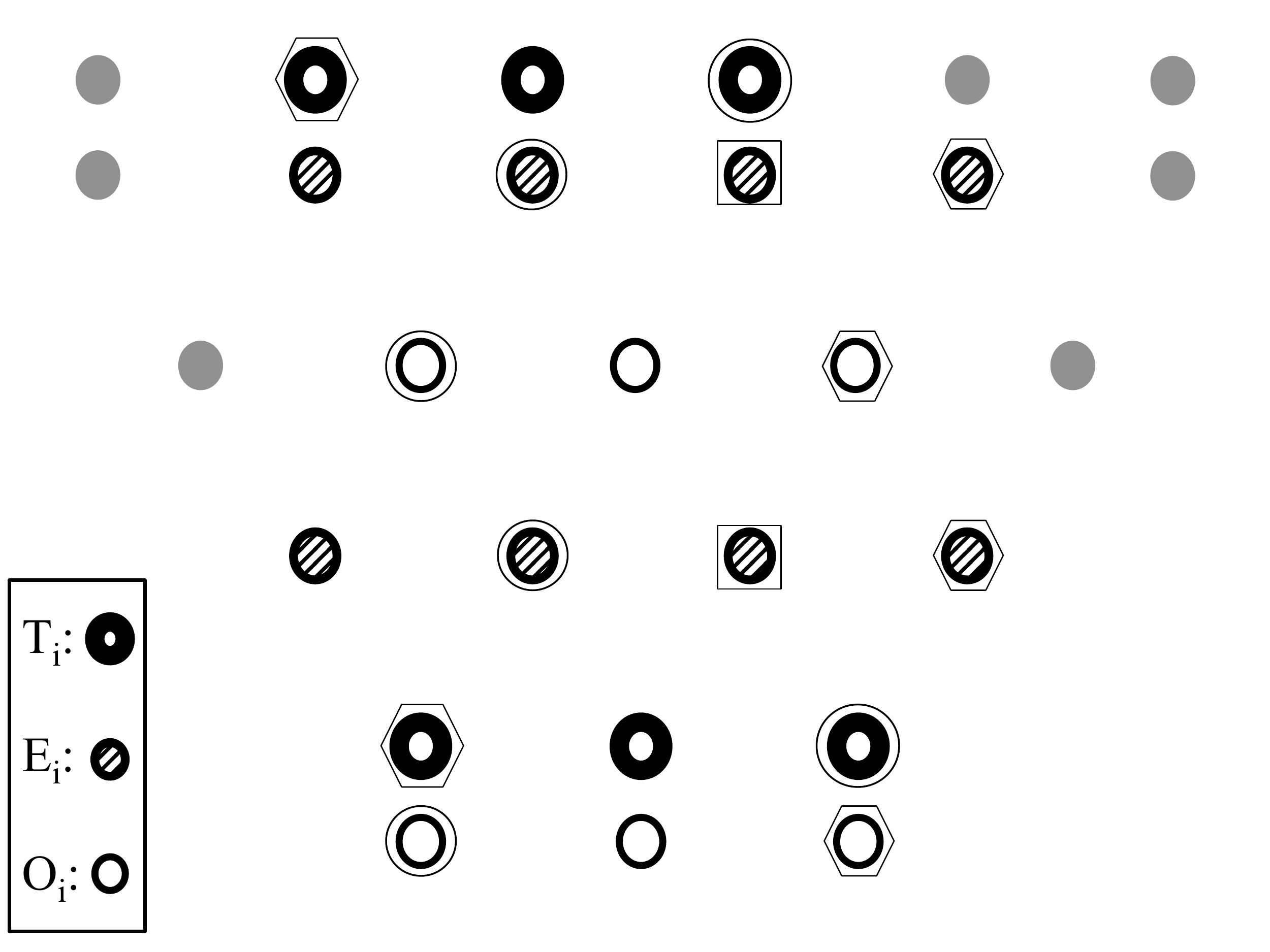}
\end{center}
\caption{Disjoint maximum antichains in $\mathcal{D}_{P_1}$ for $P_1=(6^2,5,4,3^2)$.}
\end{figure}
Now suppose that $P_2=(10,9)$ and $P=(P_2,P_1)$. Then 
$$\begin{array}{l}
\mu(P)=\min\{10, 9-2r(P_1)+10-2r(P_1))\}=10.
\end{array}$$ 

As discussed in the proof of Theorem~\ref{spread min part}, to get 10 maximum antichains in $\mathcal{D}_{P}$, we will match the top vertex of a maximum antichain of $\mathcal{D}_{P_1}$, with a vertex in level 9 or 10. Applying Lemma~\ref{strong theta<eta1}(a), with $p=4$, $g=1$, $\eta=1$, $d=0$, $\delta=3$ and $\theta=1$, we get the following matching of vertices in levels 6 and 9. (see Figure 7).
$$
\mathcal{A}_{6,9}=\left\{\begin{array}{llll}
&& \{(4,6,2),(6,9,1)\}\\
\{(2,6,1),(3,9,1)\}, & \{(3,6,1),(4,9,1)\}, & \{(4,6,1),(5,9,1)\}, & \{(5,6,1),(7,9,1)\}
\end{array} \right\}.
$$
 We then need to match the remaining vertices in level 6 with vertices in level 10. We do this using Lemma~\ref{strong theta<eta}(b) with $p=4$, $g=2$, $\eta=0$, $d=0$, $\delta=2$ and $\theta=1$ (see the dashed grouping in Figure 7). We get

$$A_{6,10}=\left\{\begin{array}{llllllllllllllllllllll}
\{(2,6,2),(3,10,1)\}, & \{(3,6,2),(4,10,1)\}&&&&&&&&&&&&&&&&&&&
\end{array}\right\}.$$
Finally, we match the vertices in level 5, with vertices left in level 10, using Lemma~\ref{strong eta<theta1}(c), with $p=3$, $g=3$, $\eta=1$, $d=0$, $\epsilon=2$ and $\theta=1$. (See Figure 7.) We get 
$$\mathcal{A}_{5,10}=\left\{\begin{array}{llllllllllllllllllllll}
&&&&&&\{(2,5,1),(5,10,1)\}, & \{(3,5,1),(6,10,1)\},& \{(4,5,1),(7,10,1)\}
\end{array}\right\}.
$$

\begin{figure}
\vspace{-.5 in}
\begin{center}
\includegraphics[scale=.35]{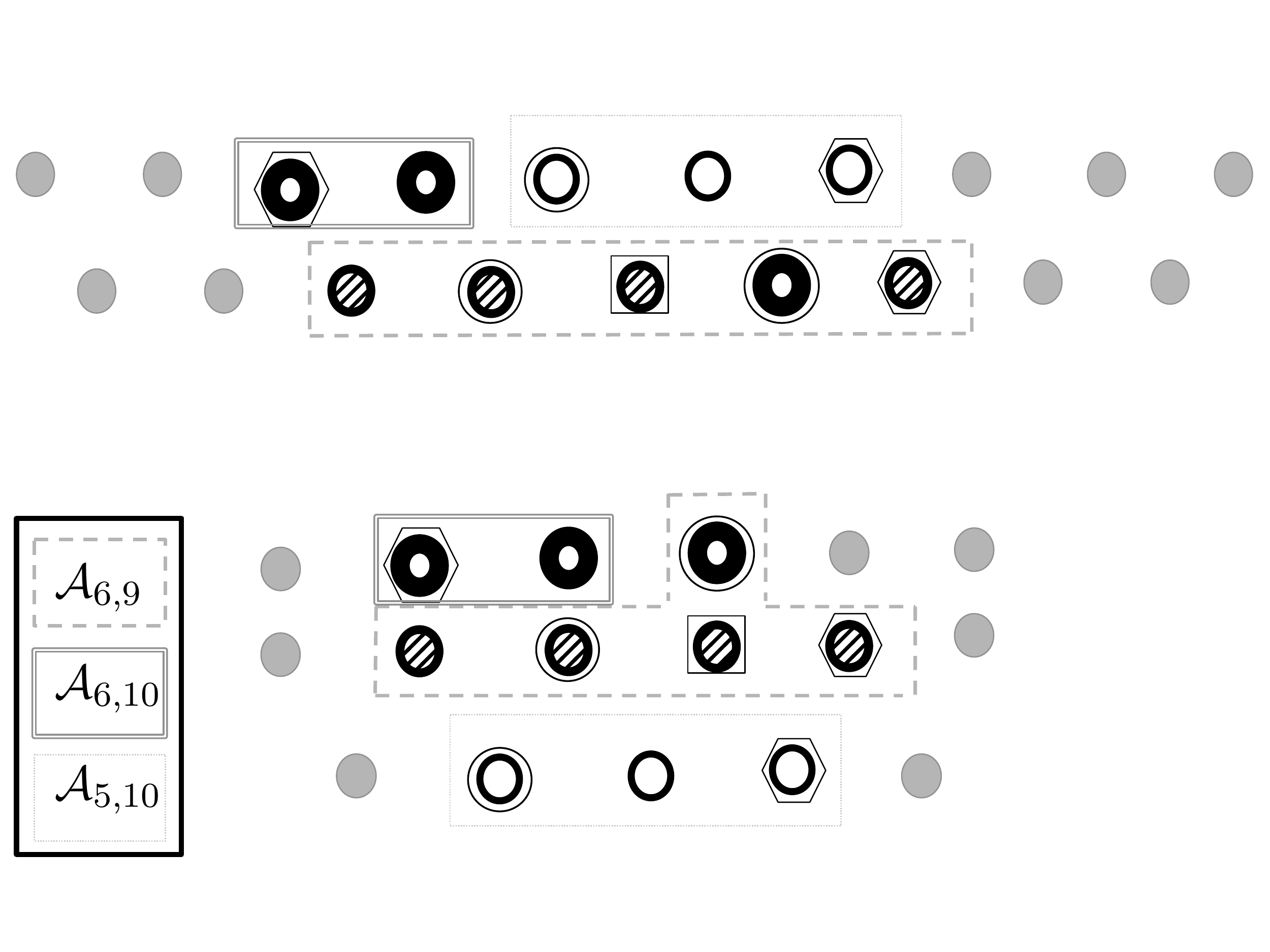}
\vspace{-.5 in}
\end{center}
\caption{Matching antichains in $\mathcal{D}_{P_1}$ with antichains in $\mathcal{D}_{P_2}$.}
\end{figure}

Thus we obtain the following 10 disjoint antichains, each of length 3, in $\mathcal{D}_P$:
$$\begin{array}{ll}
\{(1,3,1),(2,5,1),(5,10,1)\}, 
&\{(2,3,1),(3,5,1),(6,10,1)\}, \\ \{(3,3,1),(4,5,1),(7,10,1)\},&

\{(1,4,1),(2,6,1),(3,9,1)\}, \\ \{(2,4,1),(3,6,1),(4,9,1)\},& \{(3,4,1),(4,6,1),(5,9,1)\}, \\ \{(4,4,1),(5,6,1),(7,9,1)\},&

\{(1,3,2),(2,6,2), (3,10,1)\}, \\\{(2,3,2),(3,6,2),(4,10,1)\},&
\{(3,3,2),(4,6,2),(6,9,1)\}.

\end{array}$$

\end{example}

\bigskip

\begin{corollary} \label{min part}
For any partition $P$, $\min(\lambda(P))=\min(\lambda_U(P))=\mu(P).$
\end{corollary}

\begin{proof}
Combine Theorem~\ref{Oblak min} and Theorem~\ref{spread min part}. 
\end{proof}

\section{Jordan partition of the generic commuting nilpotent matrix}\label{Qmin}

Let $\mathcal{D}$ be a finite poset with cardinality $n$. Recall that the {\it incidence algebra} $\mathcal{I}(\mathcal{D})$ of $\mathcal{D}$ is defined as the subalgebra of $\mathcal{M}at_n(\k)$ consisting of all matrices $M$ with $M_{ij} \neq 0$ if and only if $i\leq j$ in $\mathcal{D}.$ If $\mathcal{D}$ has no loops, a nilpotent element $M$ of $\mathcal{I}(\mathcal{D})$ is called generic if for all $i<j$ the entries $M_{ij}$ are independent transcendentals over the field $\k$ (see \cite{Britz-Fomin}). 

In \cite{Saks} and \cite{Gansner},  M. Saks and E. Gansner independently proved that if char$\k=0$ then $\lambda(\mathcal{D})$ is the Jordan partition of a generic matrix in the incidence algebra $\mathcal{I}(\mathcal{D})$.

The same proofs (also see \cite[Theorem 6.1]{Britz-Fomin}) in fact show that over an infinite field $\k$, the Jordan partition of any nilpotent matrix $M\in \mathcal{I}(\mathcal{D})$ is dominated by $\lambda(P)$ (see \cite[Theorem B Equation (1.7)]{IK}). Now if we fix a nilpotent matrix $B \in \mathcal{M}at_n(\k)$ in Jordan canonical form with Jordan partition $P$, then a generic element of $\mathcal{U}_P$, is a nilpotent matrix in the incidence algebra $\mathcal{I}(\mathcal{D}_P)$ of the poset $\mathcal{D}_P$ (see \cite[\mbox{Equation } 2.18]{BIK}). Thus \begin{equation}\label{Qlambda}
Q(P)\leq \lambda(P).\end{equation}

\begin{theorem}\label{Q min part}
Let $\k$ be an infinite field and $B \in \mathcal{M}at_n(\k)$ a nilpotent matrix with Jordan partition $P$. Then $\min(Q(P))=\min(\lambda_U(P))=\min(\lambda(P))=\mu(P).$
\end{theorem} 
\begin{proof}
First note that by Equation~\ref{Qlambda} above and \cite[Theorem 3.9]{IK}, we have \begin{equation}\label{<<}\lambda_U(P) \leq Q(P) \leq \lambda(P).\end{equation} Therefore, since by Proposition~\ref{number of parts}, the number of parts of $\lambda_U(P)$, $Q(P)$ and $\lambda(P)$ are the same, we get $\min(\lambda(P)) \leq \min(Q(P)) \leq \min{\lambda_U(P)}$ (see Definition~\ref{dominance order}).
So the proof is complete by the equality $\min(\lambda(P))=\min{\lambda_U(P)}$ of Corollary~\ref{min part}.
\end{proof}

\begin{remark}
Theorem~\ref{Q min part} and Definition~\ref{mu defn} of $\mu$ give an algorithm to find the smallest part of $Q(P)$ for any partition $P$.
\end{remark}

\begin{corollary}\label{need gap}
Let $P$ and $P'$ be two partitions such that the largest part of $P$ is smaller than the smallest part of $P'$ minus 1. Then $\min(Q(P'\, \cup \, P))=\min\{\mu(P), \mu(P'-2r(P))\}.$

\end{corollary}
\begin{proof}
This is a consequence of Theorem~\ref{Q min part} and Lemma~\ref{mu union}.
\end{proof}

\begin{remark}
The equality in Corollary~\ref{need gap} does not hold if  the largest part of $P$ is not smaller than the smallest part of $P'$ minus 1. 

For example, let $P=(2,1)$ and $P'=(3^2)$. Then $\mu(P'\, \cup \, P)=1$ while $\mu(P)=3$ and $\mu(P'-2r(P))=\mu((1^2))=2$.
\end{remark}

Next we give an explicit formula for $Q(P)$, when it has at most three parts.

We first recall the following result of P. Oblak proved over real numbers and then extended to an arbitrary infinite field by A. Iarrobino and the author.

\begin{proposition}\label{index} (See \cite[Theorem 13]{Oblak} and \cite[Corollary 3.10]{IK})
Let $P=(\ldots, p^{n_p}, \ldots)$ be a partition of $P$. The largest part--the index-- of $Q(P)$, $i(Q(P))$ is $$i(Q(P))=\max\{an_a+(a+1)n_{a+1}+2\sum_{p>a+1}n_p\, |\, a\in \mathbb{N}\}.$$
\end{proposition}

Note that $i(Q(P))$ is the length of the longest simple $U$-chain in $\mathcal{D}_P$ (see Equation~(\ref{simple})). So if $r(P)=1$, {\it i.e.} $P$ is an almost rectangular partition, then$Q(P)=\lambda(P)=\mbox{Ob}(P)=(n)$. If $r(P)=2$, then by Proposition~\ref{index} we have $Q(P)=\lambda(P)=(i(Q(P)), n-i(Q(P))).$

\begin{corollary}\label{3 parts}
Let $P$ be a partition of $n$. If $r(P)=3$, then $$Q(P)=(i(Q(P)), n-i(Q(P))-\mu(P),\mu(P)).$$
\end{corollary}
\begin{proof}
Recall that $\mu(P)$ defined by Definition~\ref{mu defn} is, by Theorem~\ref{Q min part}, the smallest part of $Q(P)$. So the proof is clear by  Proposition~\ref{index}. 
\end{proof}

\begin{example}\label{hey}
Let $P=(8^2,7,6,5^2, 3,2^4)$. Then $r(P)=3$ and $Q(P)$ has 3 parts. 

By Proposition~\ref{index}, $i(Q(P))=|U_{\{2,3\}}|=|U_{\{7,8\}|}=23.$

We can also write $P=P_2\, \cup \,P_1$ where $P_2=(8^2,7,6,5^2)$ and $P_1=(3,2^4)$ are both spreads. By Definition~\ref{mu defn}, we have
$$\begin{array}{ll}
\mu(P_1)=11,\\
\mu(P_2-2)=10,\mbox{ (See Example }\ref{min 10}), \mbox{ and} \\
\mu(P)=\min \{\mu(P_2-2),\mu(P_1)\}=10.
\end{array}$$
Therefore $Q(P)=(23, 17,10).$

\end{example}

The enumeration of antichains in $\mathcal{D}_P$ and their comparison to parts of $\lambda_U(P)$ can also lead to explicit formulas for $Q(P)$ when it has more than three parts. The following statement gives one such formula for special families of partitions $P$ with $r(P)=4$. 
\begin{proposition}\label{4 parts}
Let $P'$ be a partition such that $Q(P')$ has three parts, and let $P''$ be an almost rectangular partition. Suppose that the largest part of $P''$ is smaller than the smallest part of $P'$ minus 1. If $ \mu(P'') \leq \mu(P'-2)$ then $Q(P'\, \cup \, P'')=Q(P')\, \cup \, Q(P'').$
\end{proposition}

\begin{proof}
Suppose that $P'$ is a partition of $n$ and $P''$ is a partition of $m$. Let $i=i(Q(P'))$ and $\mu=\mu(P')$. Then by Corollary~\ref{3 parts} we have $\lambda(P')=Q(P')=\lambda_U(P')=(i, n-i-\mu, \mu).$ Also, since $P''$ is an almost rectangular partition, we have $Q(P'')=(m).$ 

By Lemma~\ref{mu union}, we have $\mu(P'\, \cup \, P'')=m$. By Definition~\ref{lambda U} of $\lambda_U$, in this case we have $\lambda_U(P'\, \cup \, P'')=(i, n-i-\mu, \mu,m).$ We also have $\lambda_U(P'\, \cup \,P'') \leq \lambda(P'\, \cup \,P'')$.

On the other hand by Theorem~\ref{spread min part}, $\mathcal{D}_{P'}$ contains  $\mu$ disjoint antichains of length 3, and $\mathcal{D}_{P'\cup P''}$ contains $m$ disjoint antichains of length 4. Thus if we remove elements of all those length-4 antichains from $\mathcal{D}_{P'\cup P''}$ we can still find $\mu-m$ disjoint antichains of length 3 in $\mathcal{D}_{P'\cup P''}$. Let $\lambda(P'\, \cup \, P'')=(i,a,b,m).$ By Theorem~\ref{Greene}, we must have $b\geq \mu$. Thus $i+a=n-b\leq n-\mu=i+(n-i-\mu)$. Since $\lambda_U(P'\, \cup \,P'')$ is dominated by $\lambda_U(P'\, \cup \,P'')$, this implies $\lambda_U(P'\, \cup \,P'') =\lambda(P'\, \cup \,P'')$. Thus the proof of this part is complete by Equation~(\ref{<<}). 

\end{proof}

\begin{example}
Let $P=(10^2,9,8,7^2, 5,4^4,2^3,1^4)$. Then $r(P)=4$ and $Q(P)$ has 4 parts. Let $P'=(10^2,9,8,7^2, 5,4^4)$ and $P''=(2^3,1^4)$. Note that $P=P'\,\cup \, P''$. We have $\mu(P'')=10$ and $\mu(P'-2)=\mu(\,(8^2,7,6,5^2, 3,2^4)\,)=10$ by Example~\ref{hey}. Thus, by Proposition~\ref{4 parts} we have $Q(P)=(Q(P'),Q(P'')).$

By Corollary~\ref{3 parts}, $Q(P')=(29,22,13)$. Thus $Q(P)=(29, 22, 13, 10).$

\end{example}

\bibliography{ref}
\bibliographystyle{abbrv}

\bigskip

{\sc Department of Mathematics, Union College, Schenectady, NY 12308} \par
{\it E-mail Address}: khatamil@union.edu
\end{document}